\documentclass{amsart}
\usepackage{amssymb,amsmath, amsthm,latexsym}
\usepackage{graphics}
\usepackage{psfrag}
\usepackage{tikz-cd}
\usepackage{amscd}
\usepackage{graphicx}
\newcommand{\cal}[1]{\mathcal{#1}}
\theoremstyle{plain}
\newtheorem{theo}{Theorem}

\newtheorem*{question}{Question}

\newtheorem{lemma}{Lemma}[section]
\newtheorem{theorem}[lemma]{Theorem}
\newtheorem{proposition}[lemma]{Proposition}
\newtheorem{corollary}[lemma]{Corollary}
\theoremstyle{definition}

\newtheorem{definition}[lemma]{Definition}
\newtheorem{remark}[lemma]{Remark}
\newtheorem{example}[lemma]{Example}
\let\egthree=\phi
\let\phi=\varphi
\let\varphi=\egthree

\begin{document}
\title{Signature of surface bundles and bounded cohomology}
\author{Ursula Hamenst\"adt}
\thanks{Partially supported by the Hausdorff Center Bonn\\
AMS subject classification:20J06, 55N10, 57M07}
\date{November 11, 2020}


\begin{abstract}
Extending a result of Morita, we show that all tautological classes
of the moduli space of genus $g$ curves are bounded. 
As an application, we obtain that for a surface bundle
$E\to B$ over a closed surface, the
Euler characteristic $\chi(E)$ and the signature $\sigma(E)$
are related by $3 \vert \sigma(E)\vert \leq \vert \chi(E)\vert$. 
\end{abstract}

\maketitle

\section{Introduction}

A \emph{surface bundle over a surface} is a closed oriented 
$4$-dimensional manifold $E$ which admits a fibration
$\Pi:E\to B$ onto a closed oriented surface $B$ of genus $h\geq 0$, with
fiber $S_g$ a closed surface of genus $g\geq 1$.

Natural topological invariants of such surface bundles
are the \emph{Euler characteristic} $\chi(E)$
and the \emph{signature} $\sigma(E)$.
The Euler characteristic $\chi(E)$ satisfies
\[\chi(E)=(2g-2)(2h-2).\]

The \emph{Miyaoka inequality} states that if $E$ is a \emph{complex}
surface bundle or, more generally, an arbitrary compact complex surface, then
$\vert 3\sigma(E)\vert \leq \vert \chi(E)\vert$, with equality if and only
if $E$ is a quotient of the ball. Kapovich \cite{Ka98} showed that
no surface bundle over a surface is a quotient of the ball and hence
we obtain that 
$3\vert \sigma(E)\vert < \vert \chi(E)\vert$ for all complex surface bundles
over a surface. 

Kotschick \cite{Ko98} proved that the inequality
$3\vert \sigma(E)\vert < \vert \chi(E)\vert$ also holds true for 
smooth surface bundles $E$ over surfaces which 
admit an Einstein metric, and that we always have 
$2\vert \sigma(E) \vert \leq \vert \chi(E)\vert$.
Our first main result 
extends the Miaoka inequality to all surface bundles over surfaces and confirms a
conjecture of Kotschick \cite{Ko98}. 

\begin{theo}\label{main}
Let $\Pi:E\to B$ be a surface bundle over a surface; then
\[ 3\vert \sigma(E)\vert \leq \vert \chi(E)\vert.\]
\end{theo}

For the proof of Theorem \ref{main} we use that 
the signature $\sigma(E)$ of a surface bundle $\Pi:E\to B$
over a surface can be computed as follows. 
Let us denote by ${\cal M}_g$ the \emph{moduli space of curves of genus $g$}.
The \emph{universal curve} $\Upsilon:{\cal C}\to {\cal M}_g$ is the fiber bundle
over ${\cal M}_g$ whose fiber over a point $x\in {\cal M}_g$ is just the 
Riemann surface $x$. There exists a \emph{classifying map}
$f:B\to {\cal M}_g$ so that $E=f^*{\cal C}$. 

The \emph{tautological ring} of the moduli space
${\cal M}_g$ is the subring of $H^*({\cal M}_g,\mathbb{Q})$ which
is generated by the \emph{Mumford Morita Miller classes}
$\kappa_i({\cal M}_g)\in H^{2i}({\cal M}_g,\mathbb{Q})$.
These classes are non-zero precisely if $i\leq g-2$ \cite{M87,Lo95}.
In particular, we have $\kappa_1\not=0$ for $g\geq 3$.

The tangent bundle $\nu$ of the fibers of the 
universal curve is a holomorphic line bundle on ${\cal C}$,
called the
\emph{vertical tangent bundle} in the sequel.
Thus its first Chern class $c_1(\nu)\in H^2({\cal C},\mathbb{Q})$ 
is defined, and we have
\[\kappa_i=\Upsilon_*(c_1(\nu)^{i+1})\]
where $\Upsilon_*$ is the Gysin map obtained by integration over the fiber.
It then follows from Hirzebruch's signature theorem that
\[3\sigma(E)=c_1(\nu)\cup c_1(\nu)(E)=f^*\kappa_1(B)\]
where we also write $\nu$ for the vertical tangent
bundle of $E$, and we use
naturality of Chern classes under pull-back
(see \cite{H12} for a more detailed discussion). 

As ${\cal M}_g$ is a classifying space (in the orbifold sense) for 
the \emph{mapping class group} ${\rm Mod}(S_g)$ of the surface $S_g$, 
its rational cohomology is just the rational group cohomology $H^*({\rm Mod}(S_g),\mathbb{Q})$ 
of ${\rm Mod}(S_g)$. 
For a countable group $\Gamma$ we can also define the \emph{bounded} cohomology
$H_b^*(\Gamma,\mathbb{Q})$ by requiring that all cocycles and coboundaries 
are bounded functions. 
There is a natural
comparison homomorphism
\[H_b^*(\Gamma,\mathbb{Q})\to H^*(\Gamma,\mathbb{Q})\] which is in general neither injective nor
surjective.

To summarize what is known about the bounded cohomology of the mapping class group
${\rm Mod}(S_g)$, let us consider
the \emph{Hodge bundle}
${\cal H}\to {\cal M}_g$, which is the holomorphic
vector bundle over ${\cal M}_g$ (in the orbifold sense) 
whose fiber over a point $x\in {\cal M}_g$ 
equals the $g$-dimensional vector space of holomorphic one-forms on $x$.
The Hodge bundle is the pull-back of the tautological rank $g$ holomorphic vector
bundle 
over the \emph{moduli space of principally polarized abelian varieties} 
$Sp(2g,\mathbb{Z})\backslash Sp(2g,\mathbb{R})/U(g)$ by the
\emph{Torelli map}. 
The Chern classes of this bundle define
classes in the continuous group cohomology of $Sp(2g,\mathbb{R})$,
and all these classes restrict to \emph{bounded} cohomology 
classes for any lattice in $Sp(2g,\mathbb{R})$ (we refer to \cite{HO12}
for a detailed discussion of this result of Gromov and to more precise information).  
Thus by naturality under pull-back, 
the Chern classes of the Hodge bundle are bounded rational 
cohomology classes for the mapping class group 
${\rm Mod}(S_g)$.

Morita showed that
\[\kappa_{2i-1}=(-1)^i\frac{(2i)!}{B_{2i}}{\rm ch}_{2i-1}({\cal H})\]
where $B_{2i}$ is the $2i$-th Bernoulli number
and ${\rm ch}_{2i-1}({\cal H})$ denotes the component of the \emph{Chern character}
${\rm ch}({\cal H})$ of ${\cal H}$ in the cohomology group $H^{4i-2}({\cal M}_g,\mathbb{Q})$ 
 (relation (R-1) 
in Section 2 of \cite{M87}). As a consequence, the odd Mumford Morita
Miller classes are bounded. However, this reasoning does not extend 
to the even Mumford Morita Miller classes as all Chern classes of 
the Hodge bundle are polynomials in the classes $\kappa_{2i-1}$ $(i\geq 1)$
(see Section 2 of \cite{M87} for a proof).

Answering a question of Morita \cite{M88}, we show

\begin{theo}\label{main0}
All Mumford Morita Miller classes are bounded. 
\end{theo}

Some evidence for this result was given by Morita \cite{M88} who showed that
the pull-backs of the classes $\kappa_i$ by a map
$f:M\to {\cal M}_g$ from 
a closed manifold $M$ all vanish 
if the image $f_*\pi_1(M)$ of the fundamental
group of $M$ in ${\rm Mod}(S_g)$
is amenable. This is a property which holds true for
bounded cohomology classes. 
A different proof of Morita's result is due to Bowden \cite{Bw17}. 

By an Euler characteristic computation of Harer and Zagier \cite{HZ86},
we know that
as $g\to \infty$, the tautological ring
is only a very small part of $H^*({\cal M}_g,\mathbb{Q})$. 
Morita conjectured (p.255 of \cite{M88})
that the entire rational cohomology of ${\cal M}_g$ 
consists of bounded classes. In view of more recent insight into bounded cohomology 
(see \cite{HO12}  for information relevant to this circle of ideas), we are instead tempted 
to ask the following

\begin{question}\label{boundedtaut}
Is it true that the image of the comparison map 
$H^*_b({\rm Mod}(S_g),\mathbb{Q})\to H^*({\rm Mod}(S_g),\mathbb{Q})$ coincides with the
tautological subring of $H^*({\rm Mod}(S_g),\mathbb{Q})$? 
\end{question}

The only evidence in this direction, however, is given by the analogy with locally symmetric
spaces. The work \cite{BF} which shows that the kernel of the comparison map
$H_b^2({\rm Mod}(S_g),\mathbb{Q})\to H^2({\rm Mod}(S_g),\mathbb{Q})$
is infinitely generated and which is
based on what was later called acylindrical
hyperbolicity of the mapping class group, 
may be seen as evidence towards the conjecture of Morita.
Applications of second bounded cohomology
of ${\rm Mod}(S_g)$ with nontrivial coefficients to rigidiy questions can be found in \cite{H08}.  

The strategy for the proof of Theorem \ref{main} and Theorem \ref{main0}
consists in
finding an alternative description
of the Mumford Morita Miller classes. Namely, 
the \emph{projective Hodge bundle} $P:{\cal P}\to {\cal M}_g$ 
is the holomorphic fiber bundle whose fiber over a complex curve $x$ is
the projectivization of the fiber of the Hodge bundle over $x$.
The projective Hodge bundle is a smooth quasi-projective variety
(in the orbifold sense). 
Let us consider the bundle $P^*{\cal C}\to {\cal P}$ 
whose fiber over a projective abelian differential $q$ is the Riemann surface $Pq$.
The zeros of a projective abelian differential $q$ define an 
effective divisor of degree $2g-2$ on the Riemann surface
$Pq$ which is dual to the cotangent bundle of $Pq$. 
We show in Section \ref{acomplex} that the union 
$\Delta$ of the supports of these divisors is a complex subvariety
of $P^*{\cal C}$, and the canonical
projection $\Delta\to {\cal P}$ is a branched cover of degree $2g-2$. 
The branch locus consists precisely of the zeros of differentials in ${\cal P}$ of order at least $2$.

Denote by $\nu^*$ the \emph{vertical cotangent bundle}
of the universal curve 
${\cal C}\to {\cal M}_g$ or its pull-back to ${\cal P}$, dual
to the vertical tangent bundle $\nu$.
Morita \cite{M88} showed that the
circle subbundle of this bundle is flat.
This implies the following. Let us denote by
${\rm Top}^+(S^1)$ the
group of orientation preserving homeomorphisms of the circle, and by  
${\rm Mod}(S_{g,1})$  the orbifold fundamental group of the universal curve. 
There is a homomorphism
$\rho:{\rm Mod}(S_{g,1})\to {\rm Top}^+(S^1)$ 
such that $c_1(\nu^*)$ equals the pull-back by $\rho$ of the bounded
Euler class of ${\rm Top}^+(S^1)$. In particular, $c_1(\nu^*)$ is bounded.

This information is used as follows. We note
in Section \ref{branchedmultisection} that the subvariety 
$\Delta\subset P^*{\cal C}$ can be used to construct for each surface bundle
$E\to M$ a cycle which is 
Poincar\'e dual to $c_1(\nu^*)$. This allows to evaluate  
the pull-back to ${\cal P}$ of the Mumford Morita Miller 
classes using the transfer map in rational cohomology 
defined by the branched cover $\Delta\to {\cal P}$
rather than the Gysin homomorphism. On the other hand, $\Delta$ can
be viewed as a section of the bundle over ${\cal P}$ whose fiber over $q$ is the space
of effective divisors of degree $2g-2$ in $Pq$.
This bundle of divisors in turn is the 
quotient of the $2g-2$-fold fiber product of $P^*{\cal C}$ by the symmetric group
$\mathfrak{S}_{2g-2}$ which allows for viewing the Mumford Morita Miller 
classes as defined by the Chern character of the vertical virtual tangent bundle of the bundle of divisors.
This reinterpretation then leads to Theorem \ref{main0}
in Section \ref{mmm}, and Theorem \ref{main} is established in
Section \ref{cocycle}.

\bigskip
\noindent
{\bf Acknowledgement:} I am grateful to Matteo Costantini 
and Daniel Huybrechts for useful discussions.

\section{A complex vector bundle over the projectivized Hodge bundle}
\label{acomplex}

Let $S_g$ and $S_{g,1}$ be a closed surface of genus $g\geq 2$ and a closed 
surface of genus $g\geq 2$ with a single marked point, respectively. 
The moduli space ${\cal M}_g$ of curves of 
genus $g\geq 2$ is a complex orbifold. 
The moduli space 
of genus $g$ curves  with a single marked point
(puncture) is the 
\emph{universal curve} 
\[\Upsilon:{\cal C}\to {\cal M}_g,\] 
a fiber bundle (in the orbifold sense)
over ${\cal M}_g$ 
whose fiber over the point $X\in {\cal M}_g$ 
is just the complex curve $X$. 

The moduli spaces ${\cal M}_g$ and ${\cal C}$ are quotients of the
\emph{Teichm\"uller spaces} ${\cal T}_g$ and 
${\cal T}_{g,1}$ of \emph{marked} complex curves of genus $g$ and of 
marked complex curves 
of genus $g$ with one marked point, respectively, under the 
mapping class groups
${\rm Mod}(S_g)$ of $S_g$ and ${\rm Mod}(S_{g,1})$ of $S_{g,1}$.
The marked point forgetful map induces a surjective (see \cite{FM12} 
for a proof) homomorphism
$\Theta:{\rm Mod}(S_{g,1})\to {\rm Mod}(S_g)$. This homomorphism fits into the
\emph{Birman exact sequence}
\begin{equation}
 \label{birman}
1\to \pi_1(S_g)\to {\rm Mod}(S_{g,1})\xrightarrow{\Theta} {\rm Mod}(S_g)\to 1.
\end{equation}

The \emph{Hodge bundle} 
is the holomorphic
(orbifold) vector bundle ${\cal H}\to {\cal M}_g$
whose fiber over a Riemann surface $X$ equals the
$g$-dimensional complex vector space of
holomorphic one-forms on $X$.
The moduli space of abelian differentials for $S_g$
is the complement of the zero section in ${\cal H}$. 
Its quotient under the natural action of $\mathbb{C}^*$ by scaling is the 
projective Hodge bundle
\[P:{\cal P}\to {\cal M}_g\]
which is a smooth complex orbifold. 
By this we mean the quotient of a complex manifold by a finite
group of biholomorphic automorphisms. It also is a quasi-projective
variety. 
The goal of this section is to establish some first properties of the 
projective Hodge bundle and the locus $\Delta$ of the zeros of
the projective abelian differentials in the pull-back
$P^*{\cal C}\to {\cal P}$.
 



A holomorphic one-form on a Riemann surface $X$ of genus $g$ 
has precisely $2g-2$ zeros
counted with multiplicity. Each stratum of ${\cal P}$ of 
projective differentials with the same number
und multiplicities of zeros is a smooth complex orbifold. 
Strata need not be connected, and they
are locally closed, but in general not closed subsets
of ${\cal P}$. We denote by $\overline{\cal Q}$ the set theoretic closure of 
a component ${\cal Q}$ of a stratum. This is a subvariety of ${\cal P}$ which 
is a union of components of strata. 

For $k\geq 0$ define ${\cal P}(k)$ to be the union of the strata of differentials
with precisely $2g-2-k$ zeros. Then ${\cal P}(k)$ is a suborbifold of
${\cal P}$ of codimension $k$, and these suborbifolds 
give ${\cal P}$ the structure of a complex stratified
space in the sense of the following 

\begin{definition}\label{stratified}
A \emph{stratification} of a smooth complex orbifold  $X$  is a decomposition
$X=\cup_{i=0}^mX(i)$ where for each $i$, the set $X(i)$ is a locally closed
smooth suborbifold of $X$ of codimension $i$ and such that
\[\overline{X(k)}=\cup_{i=k}^m X(i).\]
\end{definition}

Consider the pull-back 
\[\Pi:P^*{\cal C}\to {\cal P}\] 
of the universal curve ${\cal C}$ to ${\cal P}$. This is a smooth complex
orbifold, foliated into the fibers of the fibration $\Pi$. The fiber of 
$P^*{\cal C}$ over a projective differential $q$ is the Riemann surface $P(q)$.
%
Denote by
\[\Delta\subset P^*{\cal C}\]
the closed subset of $P^*{\cal C}$ whose intersection with a fiber
over a projective abelian differential $q$ consists of the zeros of $q$.

We next collect some information on the set $\Delta$.
To this end call a closed 
subvariety $Y$ of codimension one of a smooth complex 
variety $X$ a \emph{local complete intersection}
if the ideal sheaf ${\cal F}_Y$ of $Y$ in $X$
can be locally generated by a single element at every point (see p.185 of \cite{Ha77}). 


\begin{proposition}\label{localcomplete}
  $\Delta$ is a local complete intersection
  subvariety of $P^*{\cal C}$. 
\end{proposition}
\begin{proof}
The Hodge bundle
  $P_0:{\cal H}\to {\cal M}_g$ and the universal curve 
  $\Upsilon:{\cal C}\to {\cal M}_g$ are both fiber bundles
  over the same complex variety ${\cal M}_g$.
 Let ${\cal H}_*\subset {\cal H}$ be the complement of the zero section
of ${\cal H}$. Then the pull-back  
$P_0^*{\cal C}$ of ${\cal C}$ to ${\cal H}_*$
can be identified with the fiber 
product
\[{\cal W}=\{(q,z)\in {\cal H}_*\times {\cal C}\mid P_0(q)=\Upsilon(z)\}.\]
This is a holomorphic fiber bundle over ${\cal M}_g$.

The vertical cotangent bundle $\nu^*$ of the universal curve ${\cal C}\to {\cal M}_g$,
that is, the fiberwise cotangent bundle, is a holomorphic line bundle on ${\cal C}$.
By the definition of the complex structure on
${\cal H}$, the evaluation map $\theta:{\cal W}\to \nu^*$, defined by
$\theta(q,z)=q(z)$, is holomorphic, moreover
it is surjective since a linear system on a complex curve of genus $g\geq 2$ is base point free.
Thus the preimage of the
zero section of $\nu^*$ under this evaluation map is a complex subvariety $V$ of ${\cal W}$ 
of codimension one. 

The variety $V$ is invariant under the free holomorphic 
action of $\mathbb{C}^*$ on ${\cal W}$ by scaling the first coordinate and hence it 
descends to a complex
subvariety $\mathbb{C}^*\backslash V$ 
of codimension one in the quotient of ${\cal W}$ by this action, which is 
the fiber product of ${\cal P}$ with ${\cal C}$. This
fiber product is naturally biholomorphic to 
the pull-back $P^*{\cal C}\to {\cal P}$, and by this identification,
the subvariety $\mathbb{C}^*\backslash V$ is 
mapped to the
set $\Delta$.  We conclude that indeed,
$\Delta$ is a complex subvariety of $P^*{\cal C}$ of codimension one.

To show that $\Delta$ is a local complete intersection, 
observe that since $\nu^*$ is a holomorphic line bundle
on ${\cal C}$, the zero section of $\nu^*$ is locally the zero set of
a single holomorphic function.  Then locally the subvariety $V$ of 
${\cal W}$ is the zero set of the composition
of the holomorphic map $\theta$ with 
a holomorphic function on $\nu^*$ and hence $V$ is a local complete intersection.

As the map $\theta$ is equivariant with respect to action of $\mathbb{C}^*$ 
on ${\cal W}$ and the holomorphic action of $\mathbb{C}^*$ on $\nu^*$ 
by scaling, by taking a local cross section of the holomorphic 
fibration ${\cal W}\to P^*{\cal C}$ and restricting a local defining holomorphic
function for $V$ 
to this local cross section, we obtain that 
$\Delta$ is a local complete intersection 
as claimed in the proposition.
\end{proof}

As a corollary we obtain

\begin{corollary}\label{cohenmac}
  The variety $\Delta\subset P^*{\cal C}$ is
  Cohen-Macaulay.
\end{corollary}
\begin{proof}
  By passing to a finite orbifold cover we may assume that
  $P^*{\cal C}$ is smooth.
  By Proposition 8.23 of \cite{Ha77}, a 
 local complete intersection
 subvariety of a smooth variety over $\mathbb{C}$ 
 is Cohen-Macaulay.  Thus the corollary
 follows from Proposition \ref{localcomplete}.
\end{proof}

As we now have more precise information on the complex variety $\Delta$, 
we can study the finite morphism $\Pi\vert \Delta:\Delta\to {\cal P}$. 
The following statement serves as an illustration
of this idea, but it is not used in the sequel.
The notion of a holomorphic vector bundle is
meant in the orbifold sense:
There is a holomorphic vector bundle on a finite manifold cover of ${\cal P}$, with 
a natural action by bundle automorphisms
of the finite group of biholomorphic automorphisms which give rise to 
${\cal P}$.    

\begin{proposition}\label{miracle}
  The push-forward 
 of the restriction to $\Delta$ of any 
  holomorphic line bundle $L\to P^*{\cal C}$  
by the finite morphism $\Pi^\Delta=\Pi\vert \Delta$ 
is a holomorphic vector bundle
$\Pi^\Delta_!(L)\to {\cal P}$ of rank $2g-2$.
\end{proposition}  
\begin{proof}
The miracle flatness theorem Theorem 23.1 of \cite{Ma} (see also
Corollary 18.1 of \cite{Ei}) 
shows that  
  a finite morphism $\Pi^\Delta$ from a
  Cohen-Macaulay variety onto a smooth variety ${\cal P}$ is flat.
  Since the restriction of $\Pi$ to $\Delta$ is clearly finite
  and ${\cal P}$ is smooth, flatness now follows from Corollary \ref{cohenmac}.

  On the other hand, by Proposition 9.2.e of \cite{Ha77}, 
  the push-forward of a locally free sheaf under a finite
  flat morphism is locally free. Since the degree of the finite morphism 
$\Pi^\Delta$ equals $2g-2$, together this yields the proposition.
\end{proof}

The Chern character of a complex vector bundle is defined by
the power series 
\[\exp(t)=\sum_k \frac{t^k}{k!}.\]
Thus if $L\to M$ is any complex line bundle, then the 
Chern character of $L$ is given by
\[{\rm ch}(L)=\sum_k \frac{1}{k!}c_1(L)^k.\]

The Grothendieck Riemann Roch theorem
(see Theorem 18.3 of \cite{Fu84})
enables us to compute the 
Chern character ${\rm ch}(\Pi^\Delta_!({\cal O}_\Delta))$ 
of the push-forward of the trivial line bundle
$\Pi^\Delta_!({\cal O}_\Delta)$.
In its formulation, ${\rm td}(\Delta)$ is the Todd class of the
virtual tangent sheaf of $\Delta$ which is defined as follows.

The normal bundle $N(\Delta)$ of $\Delta$ is a smooth holomorphic 
line bundle on the regular part of $\Delta$. 
It can be extended to a holomorphic
line bundle on all of $\Delta$ by taking the quotient of the ideal sheaf of 
$\Delta$ by its square. 
The virtual tangent sheaf of $\Delta$ 
is then defined by 
$T_\Delta=TP^*{\cal C}\vert \Delta -N(\Delta)$, taken in the Grothendieck group 
$K_0(\Delta)$.
The following is a consequence of Corollary 18.3.1 of \cite{Fu84}. 

\begin{theorem}\label{grr}
Let $L\to P^*{\cal C}$ be a holomorphic line bundle; then 
\[{\rm ch}(\Pi^\Delta_!({\cal O}_ {\Delta}))\cdot {\rm td}({\cal P})=
\Pi^\Delta_*( {\rm td}(\Delta)).\]
\end{theorem}

Given that the singularites of $\Delta$ and the branch locus of 
$\Pi^\Delta$, with multiplicities, 
precisely correspond to the strata of ${\cal P}$, it should be 
possible to compute from the push-forward 
$\Pi_*^\Delta({\cal O}_{\Delta})$ of the restriction to $\Delta$ of the
trivial line bundle on $P^*{\cal C}$ the cohomology classes of the various 
components of strata in the Chow ring of ${\cal P}$. 
We do however not pursue this line of ideas in this article and refer
instead to \cite{BHS20} for more information.

\section{Computing the Poincar\'e dual of $c_1(\nu^*)$}
\label{branchedmultisection}

In this section we begin the topological study of Mumford Morita Miller classes. 
As in Section \ref{acomplex}, we consider the pull-back $\Pi:P^*{\cal C}\to {\cal P}$ of the 
universal curve ${\cal C}\to {\cal M}_g$ to the moduli space 
${\cal P}$ of projective abelian differentials. 
We show that for every surface bundle $E\to M$ over 
a smooth manifold $M$ of dimension $n\leq 2g-2$ we can use 
the hypersurface 
$\Delta\subset P^*{\cal C}$  as in Proposition \ref{localcomplete}
to construct 
a cycle in $E$ which is Poincar\'e dual to $c_1(\nu^*)$ where $\nu^*$
is the vertical cotangent bundle of $E$. We use this to find 
a geometric description of the Mumford Morita Miller classes using transfer
homomorphisms in cohomology. 

Let for the moment $X,Y$ be 
arbitrary connected locally path connected 
Hausdorff spaces and let
$\phi:X\to Y$ be an open and closed continuous map. 
The \emph{degree} of $\phi$ is defined as 
\[{\rm deg}(\phi)=\sup\{\sharp \phi^{-1}(y)\mid y\in Y\},\]
and the \emph{local degree} of $\phi$ at $x\in X$ is defined  
as 
\[{\rm deg}(\phi,x)=\inf_U\sup \{\sharp \phi^{-1}\phi(z)\cap U\mid z\in U\},\]
where $U$ ranges over the neighborhoods of $x$.
The following is taken from \cite{Ed76}.

\begin{definition}\label{finitebranched}
An open and closed continuous 
map $\phi:X\to Y$ is a \emph{finite branched covering} if 
${\rm deg}(\phi)<\infty$ and for each $y\in Y$, 
\[{\rm deg}(\phi)=\sum_{x\in \phi^{-1}(y)} {\rm deg}(\phi,x).\]
\end{definition}

The relevance for our purpose is Theorem 2.1 of \cite{Ed76}.

\begin{theorem}[Edmonds \cite{Ed76}]\label{edmonds} 
Let $f:X\to Y$ be a finite branched covering. Then there is a transfer
homomorphism 
\[\tau:H^*(X,\mathbb{Q})\to H^*(Y,\mathbb{Q})\]
such that $\tau\circ  f^*= {\rm deg}(f)\cdot 1$.
\end{theorem}

We have

\begin{lemma}\label{branchedcoveringdelta}
$\Pi\vert \Delta:\Delta\to {\cal P}$ is a branched covering of degree $2g-2$.
\end{lemma}
\begin{proof}
Let us consider a point $z\in \Delta$. Then $z$ is a zero of the 
projective abelian differential
$\Pi(z)$. Assume that the multiplicity of this zero equals $k\geq 1$. We claim that the local
degree of $\Pi\vert \Delta$ equals $k$.  As the sum of these  
multiplicities over all
points in $\Pi^{-1}(\Pi(z))$ equals $2g-2$, this is sufficient for the proof of the lemma.

Let $U$ be a compact neighborhood of $z$ in $P^*{\cal C}$ which is disjoint from the other
zeros of $\Pi(z)$, and let $q_i$ be a sequence of projective 
abelian differentials converging as $i\to \infty$  to 
$q=\Pi(z)$.  Then for large enough $i$, the number of zeros of $q_i$ contained in $U$ and 
counted with multiplicities
equals $k$. Namely, by the definition of the topology on ${\cal P}$, 
the divisors defined by 
$q_i$, that is, the zeros of $q_i$ counted with multiplicities,  converge in $P^*{\cal C}$ to 
the divisor defined by $q$.

To complete the proof of the lemma, it suffices to show that 
for $k\geq 2$ 
and any neighborhood $U$ of $z$ in $P^*{\cal C}$, 
there exists a differential $u\in \Pi(U)$ such that $\Pi^{-1}(u)\cap U$ contains at least two points.
However, this follows from the construction on p.87 of \cite {EMZ03}, see also the 
more detailed discussion in Section 2 of \cite{H20}. 
\end{proof}

Consider as before the complement ${\cal H}_*\subset {\cal H}$ of the zero section of the
Hodge bundle. This is a smooth complex orbifold. Let $\Xi:{\cal H}_*\to {\cal P}$ be 
the canonical projection. 
By Proposition \ref{localcomplete}, the pull-back of the subvariety $\Delta\subset 
P^*{\cal C}$ to $(P\circ \Xi)^*{\cal C}$ is a codimension one subvariety, denoted
again by $\Delta$ (as all our constructions are natural with respect to pull-back by 
$\Xi$, this should not lead to confusion).
Denote again by $\Pi:(P\circ \Xi)^*{\cal C}\to {\cal H}_*$ the canonical projection. The proof of 
Lemma \ref{branchedcoveringdelta} applies verbatim and shows that the
restriction $\Pi\vert \Delta:\Delta\to {\cal H}_*$ is a branched covering. Thus by Theorem \ref{edmonds}, there
is a transfer map $\tau: H^*(\Delta,\mathbb{Q})\to H^*({\cal H}_*,\mathbb{Q})$.

The following is the main result of this section. For its formulation, 
recall that the canonical projection $P\circ \Xi:{\cal H}_*\to {\cal M}_g$ induces a pull-back
homomorphism $(P\circ \Xi)^*:H^*({\cal M}_g,\mathbb{Q})\to H^*({\cal H}_*,\mathbb{Q})$.
Furthermore, the inclusion $\Delta\to (P\circ \Xi)^*{\cal C}$ induces a restriction homomorphism
$H^*((P\circ \Xi)^*{\cal C},\mathbb{Q})\to H^*(\Delta,\mathbb{Q})$. Let as before $\nu^*$ be the vertical 
cotangent bundle of $(P\circ \Xi)^*{\cal C}$.

\begin{theorem}\label{pushforward}
$(P\circ \Xi)^*\kappa_n=(-1)^{n+1} \tau(c_1(\nu^*)^n\vert \Delta)$.
\end{theorem}

Our approach to this theorem is purely topological. Our goal
is to evaluate the $k$-th Mumford Morita Miller class $\kappa_k$ on a $2k$-cycle in 
${\cal M}_g$. To simplify matters, as we are only discussing rational cohomology,
by a result of Thom 
it suffices to assume that this cycle is given by a smooth closed oriented manifold
$M$ and a continuous map $f:M\to {\cal M}_g$. The pull-back of the universal curve
${\cal C}\to {\cal M}_g$ under $f$ is then a surface bundle $E\to M$. The map 
$f$ is called a \emph{classifying map} for $E$.

Up to homeomorphism, the bundle only
depends on the homotopy class of $f$. Equivalently, it only depends on the 
 conjugacy class of the induced \emph{monodromy homomorphism}
$f_*:\pi_1(M)\to {\rm Mod}(S_g)$. As a consequence, we may choose the
map $f$ to be smooth (as maps between orbifolds). Then $E\to M$ is a
smooth fiber bundle with the additional property that each
fiber is equipped with a complex structure varying smoothly over the base.
In the remainder of this section we always assume that this
is the case. 

As we are interesting in understanding the pull-back of $\kappa_k$ to 
${\cal H}_*$, we like to lift the classifying map $f:M\to {\cal M}_g$ 
to ${\cal H}_*$, that is, we like to find a map $F:M\to {\cal H}_*$ with the property
that $P\circ \Xi \circ F=f$.
To this end we consider the sphere subbundle ${\cal S}\subset {\cal H}_*$ of ${\cal H}$
of area one abelian differentials. Its fiber is a sphere of real dimension
$2g-1\geq 3$, and it is homotopy equivalent to ${\cal H}_*$. 

There is a natural fiber preserving action by the circle
$S^1$ which multiplies an abelian differential with a complex number
of absolute value one. The quotient of ${\cal S}$ by this action 
is the projective Hodge bundle ${\cal P}$.

The following is well known, and a detailed argument is contained
on p.133 and p.134 of \cite{Fr17}.

\begin{lemma}\label{lift}
Let $M$ be a simplicial complex of dimension $n\leq 2g-2$ and let
$f:M\to {\cal M}_g$ be a continuous map. 
Then 
there exists a continuous map $F:M\to {\cal S}$ such that 
$P\circ \Xi\circ F=f$. Any two such maps are homotopic. 
\end{lemma}
\begin{proof} For each vertex $p$ of the simplicial complex 
$M$ choose 
a point $F(p)$ in the fiber of ${\cal S}$ over $p$. This defines a section 
of ${\cal S}$ over the zero-skeleton of ${\cal S}$.
Extend by induction on $k$ 
the section $F$ to the 
$k$-skeleton of $M$ as follows. 

Assume that we extended $F$ to a map on the $(k-1)$-skeleton of $M$ for some $k\geq 1$.
Let $\Delta^k$ be a $k$-simplex of $M$. Then $\Delta^k$ is contractible
and therefore the pull-back $(f\vert \Delta^k)^*{\cal S}$ 
of the bundle ${\cal S}$ to $\Delta^k$ is trivial. 
Since the fiber of ${\cal S}$ is $2g-2\geq k$-connected, the restriction of the 
section $F$ to the boundary $\partial \Delta^k$ of $\Delta^k$ 
can extended to $\Delta^k$. This yields the induction step. 

We have to show that any two such lifts $F,\Phi$ of $f$ are homotopic. 
Namely, since the fiber of ${\cal S}$ is connected, we can connect
the images under $F,\Phi$ of the vertices of $M$ by an arc in the fixed fiber. These arcs then
define a homotopy between the restrictions of $F,\Phi$ to the zero-skeleton of $M$.
By induction on $k$, this homotopy can be extended to a homotopy between
the restrictions of $F,\Phi$ to the $k$-skeleton of $M$ using the fact that the fiber of 
${\cal S}$ is $k$-connected for all $k\leq 2g-2$.
\end{proof}

Since smooth manifolds can be triangulated, 
Lemma \ref{lift} yields that if $M$ is a smooth closed manifold of dimension
$n\leq 2g-2$, then 
for a continuous map $f:M\to {\cal M}_g$ there exists a lift 
$F:M\to {\cal S}$ of $f$.  A homotopy of $F$ descends to a homotopy of $f$, and a 
homotopy of a classifying map does not change the surface bundle obtained
as pull-back of the universal curve ${\cal C}\to {\cal M}_g$.
Thus when analyzing the topology of the surface bundle $E=f^*{\cal C}$, 
we are free
to modify the map $F$ with a homotopy. As $\kappa_{g-1}=0$ \cite{Lo95}, 
it is sufficient for the
understanding of the classes $\kappa_i$ to consider cycles in ${\cal M}_g$ defined
as the projections of smooth maps $F:M\to {\cal S}$ where $M$ is a smooth 
manifold of dimension $2n\leq 2g-4$.

The strategy is as follows. We begin with showing that
by modifying $F$ with a homotopy, we may assume that 
$\Delta_E=F^*\Delta\subset F^*(P\circ \Xi)^*{\cal C}=E$ is a cycle
of dimension $2n$ so that the restriction of $\Pi_E:E\to M$ 
to $\Delta$ is a branched covering of degree $2g-2$.
If we denote by $\tau_E:H^*(\Delta_E,\mathbb{Q})\to
H^*(M,\mathbb{Q})$ the transfer homomorphism given by Theorem \ref{edmonds}, 
then we obtain a commutative diagram
\begin{equation}\label{pullbackdiagram}\begin{tikzcd}
    H^*(\Delta,\mathbb{Q}) \arrow[r,"F^*"] \arrow[d,"\tau"] &
    H^*(\Delta_E,\mathbb{Q})\arrow[d,"\tau_E"]\\
 H^*({\cal M}_g,\mathbb{Q})\arrow[r,"f^*"] &H^*(M,\mathbb{Q}).
  \end{tikzcd}
\end{equation}   
By naturality of the Chern classes under pull-back, if we denote by
$F^*(\nu^*)$ the pull-back of the line bundle $\nu^*$ under the bundle
map $E\to (P\circ \Xi)^*{\cal C}$ which is induced by the map $F$,
then we show that the image of $(-1)^{n+1}c_1(F^*\nu^*)^n$ under 
the transfer homomorphism $\tau_E$ equals the class $f^*\kappa_n$.

The remainder of this section is devoted to implement this strategy.
To this end and for the purpose of obtaining some additional
geometric information on surface bundles,
it will be convenient to use a considerably more
restricted notion of a branched covering which 
we introduce next. For its formulation, call a
simplicial complex $X$ to be \emph{of homogeneous dimension $n$}
if the $n$ is the dimension of $X$ and if furthermore every
simplex is contained in a simplex of dimension $n$.
We always identify a simplicial complex with its geometric
realization. Note that a locally compact 
simplicial complex is locally connected. 


\begin{definition}\label{finitebranchedsimp}
An open and closed surjective 
simplicial map $\phi:X\to Y$ between
locally compact connected 
simplicial complexes $X,Y$ of homogeneous dimension $n$ 
is an \emph{unfolded simplicial
branched covering of degree $m\geq 1$} if 
there exists a simplicial subcomplex $X^1\subset X$
of codimension 2 with the following properties.
\begin{enumerate}
\item $\phi^{-1}(\phi(X_1))=X_1$. 
\item
For each $x\in X-X^1$ we have 
$\sharp \phi^{-1}(\phi(x))=m$.
\item Each point $y\in \phi(X^1)$ admits a neighborhood basis
by open sets $U_i$ such that $U_i-\phi(X^1)$ is path connected for all $i$.
\end{enumerate}
\end{definition}

\begin{lemma}\label{simplicial}
An unfolded simplicial branched covering $\phi:X\to Y$ 
is a branched covering in the sense of Definition \ref{finitebranched}.
\end{lemma}
\begin{proof}
Let $X$ be a locally compact connected 
simplicial complex of homogeneous dimension $n$ and let 
$\phi:X\to Y$ be an unfolded simplicial branched covering 
of degree $m\geq 1$. 
By assumption, $\phi$ is open and closed.

Let $X^1\subset X$ be as in Definition \ref{finitebranchedsimp}. 
Then $X^1$ is a closed subset of $X$ of codimension $2$. Since 
$X$ is of homogeneous dimension $n$, the complement 
$X-X^1$ is open and dense. 
By Properties (1) and (2) in Definition \ref{finitebranched} and surjectivity of 
the map $\phi$, 
the restriction of $\phi$ to $X-X^1$ 
is an ordinary covering projection of degree $m$ onto $Y-\phi(X_1)$.  
Thus for each $y\in Y-\phi(X^1)$ and each $x\in \phi^{-1}(y)$ we have
${\rm deg}(\phi,x)=1$ and 
$\sum_{x\in \phi^{-1}(y)}{\rm deg}(\phi,x)=m$. 

Now let $y\in \phi(X^1)$ and let $\{x_1,\dots,x_s\}=\phi^{-1}(y)\subset X$.
Property (1) implies that $\{x_1,\dots,x_s\}\subset X^1$. 
For each $i$ choose an open path connected neighborhood $U_i$ of $x_i$ such that
the sets $U_i$ are pairwise disjoint. Since $\phi$ is open, the 
set $U=\cap_i\phi(U_i)$ is a neighborhood of $y$ in $Y$. 
By 
Property (3) in Definition \ref{finitebranchedsimp}, there exists an open path connected 
neighborhood $V\subset U$ of $y$ in $Y$ such that $V-\phi(X^1)$ is path connected. 
Since $\phi(X^1)$ is of codimension two and $Y$ is of homogeneous dimension $m$, 
we may assume that the closure of 
$V-\phi(X^1)$ contains $V$.

Since
the restriction of $\phi$ to $\phi^{-1}(V-\phi(X^1))$ is a finite unbranched covering of degree $m$ and  each point
$z\in V-\phi(X^1)$ has at least one preimage in each of the sets $U_i$,
we conclude that $s\leq m$.
Furthermore, as $V-\phi(X^1)$ is path connected,
$\phi^{-1}(V-\phi(X^1))$ is a disjoint 
union of $p\leq m$ path connected sets $W_j$
such that for each $j$, the restriction of 
$\phi$ to $W_j$ is an unbranched covering $W_j\to V-\phi(X^1)$. 
Now $\phi$ is closed and therefore for each $j$, 
the image $\phi(\overline{W}_j)$ of 
the closure $\overline{W_j}$ of $W_j$ contains the closure of 
$V-\phi(X^1)$, and hence it contains $V$.
In particular, $\overline{W}_j$ contains one of the points $x_i$. 
Thus the local degree of $x_i$ for the restriction of $\phi$ to the closure 
of $W_j$ is not smaller than $d_i=\sum a_j$, where $a_j$ denotes 
the degree of the covering
$W_j\to V-\phi(X^1)$, and the sum is over all $j$ so that $x_i\in \overline{W_j}$. 
Furthermore, we have $\sum_i d_i=m$.

If the local degree of $\phi$ at $x_i$ 
is bigger than $d_i$ 
then there exists a point $z\in V\cap \phi(X^1)$ and a set $W_j$ whose closure contains
$x_i$ and such that $z$ has more than $a_j$ preimages
 in $\overline{W_j}$. The above reasoning then implies that then there is a point in 
 $V-\phi(X^1)$ near $z$ with more than $a_j$ preimages in $W_j$ which is 
 impossible. Together this shows that indeed,
 $\sum_{i}{\rm deg}(\phi,x_i)=m$ as claimed. 
As $y\in \phi(X^1)$ was arbitrary, the map $\phi$ is a branched
covering as claimed in 
the lemma. 
\end{proof}

\begin{example}\label{surfacebranched}
Recall that 
a branched covering of a surface $\Sigma$ over
a surface $B$ is a finite to one 
surjective map $\phi:\Sigma\to B$ such that there exists a finite set 
$A\subset B$, perhaps empty, with the property that 
$\phi\vert \phi^{-1}(B-A)$ is an ordinary covering projection. 
Given such a branched covering $\Sigma\to B$, choose a triangulation $T$
of $B$ such that each point in the set $A$ is a vertex of the triangulation.
Then the preimage of $T$ is a triangulation of $\Sigma$ so that the interior 
of each edge and each triangle has precisely $m$ preimages where $m$ is the degree of the 
covering. Thus for these triangulations, the branched covering $\phi$ is an
unfolded  simplicial
branched covering in the sense of Definition \ref{finitebranchedsimp}. 
\end{example}

A \emph{section} of a surface bundle  
$\Pi:E\to M$ is a smooth
map $\sigma:B\to E$ such that $\Pi\circ \sigma={\rm Id}$. 
The following is well known and reported here
for completeness.
The recent article \cite{CS18} contains a  comprehensive discussion of this
and related facts. 
For its formulation, define a \emph{lift} of the monodromy
homomorphism $f_*:\pi_1(M)\to {\rm Mod}(S_g)$ to be a homomorphism
$\tilde f_*:\pi_1(M)\to {\rm Mod}(S_{g,1})$ with the property that
$\Theta \circ \tilde f_*=f_*$ where $\Theta$ is defined in (\ref{birman}).  

\begin{lemma}\label{sectionexists}
The surface bundle $E\to M$ 
has a section if and only if there is a lift
$\tilde f_*:\pi_1(M)\to {\rm Mod}(S_{g,1})$ of 
the monodromy homomorphism $f_*:\pi_1(M)\to {\rm Mod}(S_g)$. 
\end{lemma}

\begin{example}\label{cobordant}
A surface bundle over a surface may not admit a section, however it is known
that any surface bundle over a surface is cobordant to a surface bundle
which admits a section. We refer to \cite{H83} for more information.
\end{example}

We shall use the following generalization of the
notion of a section of a surface bundle $E\to M$. In its formulation, 
a simplicial map of a simplicial complex onto a smooth manifold means
a simplicial map for a suitable choice of a triangulation of $M$.
If we talk about an unfolded simplical branched covering 
$\phi:X\to Y$ where $X$ is disconnected, then we assume that 
$X$ consists of finitely many components and that the restriction of 
$\phi$ to each of these components is an unfolded simplicial branched
covering in the sense of Definition \ref{finitebranchedsimp}. The degree 
of such a map $\phi$ 
then is defined to be the sum of the degrees of the restriction of $\phi$ to each
connected component of $X$.

\begin{definition}\label{branchedmulti}
A \emph{branched multi-section of degree $d$}
of a surface bundle
$\Pi:E\to M$ is defined to be an embedding
$\phi:\Sigma\to E$ where $\Sigma$ is a 
(not necessarily connected) simplicial complex and 
such that $\Pi\circ \phi:\Sigma\to M$ is an unfolded simplicial branched covering of degree $d$.
\end{definition}

Any section of a surface bundle $E\to M$ is a branched multisection as in 
Definition \ref{branchedmulti}. A \emph{multisection} of $E$ is 
an embedding $\phi:\Sigma\to E$ so that $\Pi\circ \phi$ is 
an unbranched covering. 



\begin{lemma}\label{cycle}
Let $\Pi:E\to M$ be a surface bundle over 
an oriented closed manifold $M$ of dimension $n$ and let 
$\phi:\Sigma\to E$ be a branched multisection of 
degree $d$. Then $\phi(\Sigma)$ is a cycle in 
$E$ defining a homology class $\phi(\Sigma)\in H_n(E,\mathbb{Z})$.
\end{lemma}
\begin{proof}
By definition, if $n={\rm dim}(M)$ then $\Sigma$ is a simplicial complex
of homogeneous dimension $n$. Furthermore, 
each $n$-dimensional simplex of $M$ has precisely $d$ preimages in $\Sigma$, and
each of these preimages inherits from 
the orientation of $M$ an orientation.  These orientations are compatible
with boundary maps. As such boundary maps commute with the restriction of the
projection $\Pi$ to $\phi(\Sigma)$, 
this yields that $\phi(\Sigma)\subset E$ is indeed a cycle.
We note that Theorem \ref{edmonds} can also be used for a proof of this lemma.
\end{proof}

The vertical cotangent bundle $\nu^*$ of $E$ is the pull-back of the
vertical cotangent bundle of the universal curve ${\cal C}\to {\cal M}_g$ under
a classifying map for $E$.
The following is the main technical result of this section.

\begin{theorem}\label{branchedpoincare}
  Let $E\to M$ be a surface bundle over a smooth closed
  oriented manifold $M$ of real dimension $n\leq 2g-4$.
  Then up to homotopy, a classifying map $f:M\to {\cal M}_g$
    admits a lift $F:M\to {\cal H}^*$ with the following property.
    The preimage $\Delta_E$ of the subvariety $\Delta$ of
 $F^*\Xi^*P^*{\cal C}=E$ is the image of     
a branched multisection of degree $2g-2$ whose
homology class is Poincar\'e dual to the Chern class
$c_1(\nu^*)\in H^2(E,\mathbb{Z})$ of 
the vertical cotangent bundle $\nu^*$. 
\end{theorem}

We are now ready to show Theorem \ref{pushforward} under the assumption
that Theorem \ref{branchedpoincare} holds true.

\begin{proof}[Proof of Theorem \ref{pushforward}]
  Let $k\leq g-2$, let $M$ be a closed oriented manifold of dimension $2k$ 
  and let $f:M\to {\cal M}_g$ be a smooth map. By
  Theorem \ref{branchedpoincare}, up to changing $f$ by
  a homotopy which does not change the pull-back homomorphism
$f^*:H^*({\cal M}_g,\mathbb{Q})\to H^*(M,\mathbb{Q})$,  
we may assume that there exists a lift $F:M\to {\cal H}_*$ of
$F$ so that the pull-back
$\Delta_E$ of $\Delta\subset
(P\circ \Xi)^*{\cal C}$ by $F$ is a branched multisection of
$E\to M$ whose homology class is Poincar\'e dual to $c_1(\nu^*)$.

Thus if we denote by $[E],[M]$ the fundamental
class of $E,M$ then
by the definition of $\kappa_k$, we have
\begin{equation}\label{firstidentity}
  f^*\kappa_k[M]=(-1)^{k+1}c_1(\nu^*)^{k+1}[E]=
(-1)^{k+1}  c_1(\nu^*)^k(\Delta_E).\end{equation}
On the other hand, the restriction of the map
$\Pi_E:E\to M$ to $\Delta_E$ is a branched covering of degree $2g-2$.
Thus if we let as before 
$\tau_E:H^*(\Delta_E,\mathbb{Q})\to H^*(M,\mathbb{Q})$
the transfer map, then we have
\begin{equation}\label{secondidentity}
  c_1(\nu^*)^k(\Delta_E)=\tau_E(c_1(\nu^*)^k)[M].\end{equation}
As we evaluate a top degree cohomology class on fundamental cycles,
this identity is essentially the definition of the transfer map. 

The equations (\ref{firstidentity}) and (\ref{secondidentity})
and naturality of the construction shows that the map $F$ induces
a commutative diagram with the properties given in
the diagram (\ref{pullbackdiagram}). As $M$ was an arbitrary
smooth manifold of dimension $2k$ and $f:M\to {\cal M}_g$
was an arbitrary continuous map, this completes the proof of
Theorem \ref{pushforward}.
\end{proof}

We are left with the proof of Theorem \ref{branchedpoincare},
which is established with a general position argument.
Namely, 
the complex stratification ${\cal P}=\cup_k{\cal P}(k)$ lifts to an 
$S^1$-invariant stratification 
${\cal S}=\cup_k{\cal S}(k)$, with strata of even real codimension. 
The stratum ${\cal S}(k)$ is a locally closed 
smooth suborbifold of ${\cal S}$, and 
the closure of ${\cal S}(k)$ equals 
\[\overline{{\cal S}(k)}=\cup_{m\geq k}{\cal S}(m).\] 
Thus the strata define a stratification
in the sense of Chapter 2 of \cite{EM02}.

Our next step is to put the lift $F$ of the classifying map $f$ 
into general position with respect to
the stratification of ${\cal S}$. To this end define a smooth map
$G:M\to {\cal S}$ to be \emph{transverse} to the 
stratification ${\cal S}=\cup_k{\cal S}(k)$ if $G$
is transverse to each stratum, 
viewed as a smooth suborbifold of ${\cal S}$. The following is 
an immediate consequence of Thom's transversality theorem.
We refer to Chapter 2 of \cite{EM02} for details on maps transverse
to stratified manifolds (or orbifolds). To avoid technical issues
arising from the existence of orbifold points, here and in the sequel
we pass to a finite manifold cover of ${\cal S}, {\cal P}, {\cal C}, {\cal M}_g$
whenever we use results from differential topology or complex analysis
and observe that all constructions can be made equivariant with respect 
to a finite group of structure preserving maps.

\begin{proposition}\label{transverse}
The map $F$ can be changed with a homotopy to a smooth map which is
transverse to the strata
of the stratification of ${\cal S}$. In particular,
the preimages of the strata ${\cal S}(k)$ define a stratification of 
$M$.
\end{proposition} 
\begin{proof}
All we need to do is to observe that Thom's transversality theorem
can be applied directly to the situation at hand, 
as explained in Chapter 2 of \cite{EM02}.
Namely, write $W=M\times {\cal S}$. Then a 
continuous or smooth map 
$F:M\to {\cal S}$ is just a continuous or smooth section of the bundle $W\to M$.

Any continuous map from a closed smooth manifold to another smooth manifold
is homotopic to a smooth map, so we may assume that $F$ is smooth. 
The products $M\times {\cal S}(k)$ define
a stratification of $M\times {\cal S}$. By Theorem 2.3.2 of \cite{EM02},
the map $F$, viewed as a section of $W$, can be modified with a generic
small homotopy to be transverse to this stratification. This is 
just the statement of the proposition.
\end{proof}

Let $F:M\to {\cal S}$ be smooth and 
transverse to the stratification ${\cal S}=\cup_k{\cal S}(k)$. 
Then $\Pi_E:E=F^*(P\circ \Xi)^*{\cal C}\to M$ is a surface bundle over
$M$. As before, we denote by $\Delta_E$ the pull-back of the subvariety
$\Delta$ of $(P\circ \Xi)^*{\cal C}$. Then $\Delta_E$ is a closed subset
of $E$, and 
for each $x\in M$, the intersection of $\Delta_E$ with
$\Pi_E^{-1}(x)$ is the set of zeros of
the abelian differential $F(x)$.
The following proposition
concludes the proof of Theorem \ref{branchedpoincare}. 

\begin{proposition}\label{poincaredual}
$\Delta_E\subset E$ is the image of a branched multisection of $E\to M$ 
of degree $2g-2$ 
whose homology class is Poincar\'e dual to $c_1(\nu^*)$.
\end{proposition}
\begin{proof} By construction, for each 
$k$ the set $M(k)=F^{-1}{\cal S}(k)$ of all points whose image under 
$F$ is contained in the suborbifold ${\cal S}(k)$ of ${\cal S}$ is 
a submanifold of codimension $2k$, and $\overline{M(k)}=\cup_{u\geq k}M(u)$.

Since $M=\cup_kM(k)$ is a smooth stratified manifold, with strata 
$M(k)$ of codimension $2k$, we can find 
a triangulation $T$ of $M$ such that for each $k$, the subset 
$\overline{M(k)}$ is a subcomplex of $T$ of codimension $2k$. 
In particular, 
if $\sigma$ is a simplex all of whose vertices are contained in $\overline{M(k)}$, then
$\sigma\subset \overline{M(k)}$.
We refer to \cite{J83} for an explicit statement, for the history of the problem and  
for references to earlier work from which this statement can also be extracted
(perhaps less directly).

By construction, we have
\[M(k)=\{z\in M\mid \sharp \Pi^{-1}(z)\cap \Delta_E=2g-2-k\}.\]
More precisely, for all $k$ the restriction of 
$\Pi$ to $\Delta_E(k)=\Pi^{-1}(M(k))\cap \Delta_E$ 
is an unbranched covering of degree
$2g-2-k$. As a consequence, the preimage of a simplex $\sigma\subset T$
consists of precisely $2g-2-k$ simplices with disjoint interior
if $k$ is the maximum of all numbers $\ell\geq 0$ such that each
vertex of $\sigma$ is contained in $\overline{M(\ell)}$,
and the union of these simplices defines the structure of a simplicial
complex on $\Delta_E$ of homogeneous dimension $n$. Clearly
$\Delta_E$ is locally compact. 

If we define $\Delta_E^1\subset \Delta_E$
to be the subcomplex of all simplices which project into
$\overline{M(1)}$ then as $\overline{M(1)}$ is a subcomplex
of $M$ of codimension $2$, the set $\Delta_E^1$ is a subcomplex of $\Delta_E$ 
of codimension $2$.
Furthermore, as a point in $M(0)$ has precisely $2g-2$ preimages in $\Delta_E$, 
the restriction of $\Pi$ to
$\Delta_E-\Delta_E^1$ is an unbranched covering of degree $2g-2$.
This verifies the first property in Definition \ref{finitebranchedsimp}.

If $x\in \overline{M(1)}$ is arbitrary, then since 
$\overline{M(1)}$ is of codimension two in $M$, 
there exists a neighborhood basis 
for $x$ in $M$ 
consisting of sets $U_i$ such that $U_i-\overline{M(1)}$ is path
connected. 
Thus $\Pi\vert \Delta_E:\Delta_E\to M$ fulfills the requirements in 
Definition \ref{finitebranchedsimp} and hence it 
is an unfolded simplicial
branched covering.

By Theorem \ref{edmonds} and Lemma \ref{cycle}, the preimage of 
the triangulation $T$ under the map  
$\Pi\vert \Delta_E$ equips $\Delta_E$ with the structure of a cycle 
of dimension $n$. This cycle then defines a homology class
$[\Delta_E]\in H_{n}(E,\mathbb{Z})$. 
As the total space $E$ of the surface bundle 
is a closed oriented manifold, it is a Poincar\'e duality space. 
The Poincar\'e dual of $[\Delta_E]$ is a class in $H^2(E,\mathbb{Z})$. 

To compute this class, 
call a point in $\Delta_E$ regular if it is not a zero of order at least $2$ for 
a differential in the set $F(M)$. The union of the regular points is
an open submanifold of $E$ 
whose complement is of codimension four. 
Namely, since $\cup_{k\geq 1}M(k)$ is 
of codimension two in $M$ and since $\Delta_E$ intersects each
fiber of $E$ in a finite set, the set of singular points in $\Delta_E$,
that is, points which are
not regular, is a simplicial subcomplex of $\Delta_E$ of codimension $2$
and hence it has codimension $4$ in $E$. 
Hence if $B$ is a closed surface and if $\alpha:B\to E$ is a smooth map, 
then by transversality, $\alpha$ can be modified
with a small homotopy to a smooth map whose intersection with  
$\Delta_E$ consists of 
finitely many regular transverse intersection points. For transversality, observe that
the complement of the singular points of $\Delta_E$ is a smooth submanifold of $E$.
With a further smooth homotopy, we can assure that for all $u\in B$ so that
$\alpha(u)\in \Delta_E$, there is a neighborhood of $u$ in $B$ which is mapped
by $\alpha$ diffeomorphically onto a neighborhood of $\alpha(u)$ in
the fiber $\Pi^{-1}(\Pi(\alpha(u))$. 

In view of the fact that every second integral homology
class in $E$ can be represented by a smooth map from 
a closed oriented surface (for smoothness, note that every continuous
map $B\to E$ is homotopic to a smooth map), 
to compute the class in 
$H^2(E,\mathbb{Z})$ which is Poincar\*e dual to $\Delta_E$ it suffices to show
that for each such smooth map $\alpha:B\to E$
the number of such intersection
points between $\alpha(B)$ and $\Delta_E$, counted with sign and 
multiplicity,
equals the degree of the pull-back 
line bundle $\alpha^*\nu^*$ on $B$.

To this end observe that
for each $z\in E-\Delta_E$, the restriction of the 
holomorphic one-form $F(\Pi(z))$ to the tangent
space of the fiber of $E\to B$ at $z$ does not vanish
and hence it defines a nonzero element $\beta(z)$ 
of the fiber of 
$\nu^*$ at $z$. 
Associating to $z$ the linear functional $\beta(z)$ defines a 
trivialization of $\nu^*$ on $E-\Delta_E$.

 At each regular  
point $y\in \Delta_E$, 
the restriction of this trivialization
to the oriented fiber $\Pi^{-1}(\Pi(y))$ has rotation
number $1$ about $y$ with respect to a trivialization which
extends across $y$.
Namely, this is equivalent to the statement that the 
divisor on the Riemann surface $\Pi^{-1}(\Pi(y))$ 
defined by $F(\Pi(y))$ defines the holomorphic cotangent
bundle of $\Pi^{-1}(\Pi(y))$. 

Although this is well known, we provide the short argument. 
The holomorphic one-form $F(\Pi(y))$ on the surface
$\Pi^{-1}(\Pi(y))$ defines an euclidean metric on 
$\Pi^{-1}(\Pi(y))-\Delta_E$ which extends to a singular
euclidean metric on all of $\Pi^{-1}(\Pi(y))$.
As $y$ is a simple zero of  
the abelian differential $F(\Pi(y))$ by assumption,  
it is a four-pronged singular point for this
singular euclidean metric. Now let $D\subset \Pi^{-1}(\Pi(y))$ be a small
disk about $y$ not containing any other
point of $\Delta_E$, with boundary $\partial D$. 
Choose a nowhere 
vanishing vector field $\xi$ on $\partial D$ with $F(\Pi(y))(\xi)>0$.
As $y$ is a four-pronged singular point, the rotation number of 
$\xi$ with respect to a vector field which extends to a trivialization 
of the tangent bundle of $D$ equals $-1$. By duality,  
the rotation number of the trivialization of the cotangent bundle 
$\nu^*$ of  $\Pi^{-1}(\Pi(y))-\Delta_E$ defined by $F(\Pi(y))$ 
on $\partial D$ with respect to a trivialization of the cotangent bundle
which extends to all of $D$ equals $1$.

As the restriction of $\nu^*$ to $E-\Delta_E$ admits a natural trivialization,
the same holds true for the pull-back of $\nu^*$ to 
$B-\alpha^{-1}(\Delta_E)$. 
Furthermore,  by assumption on $\alpha$, 
for each $u\in B$
with $\alpha(u)\in \Delta_E$, the induced trivialization of 
the pull-back bundle $\alpha^*(\nu^*)$ 
on $B-\alpha^{-1}(\Delta_E)$ has rotation number $1$ 
or  $-1$ at $u$  with respect to a trivialization of
$\alpha^*\nu^*$ on a neighborhood $u$.  Here the sign 
depends on whether $\alpha$, viewed as a diffeomorphism from 
a neighborhood of $u$ in $B$ 
onto a neighborhood of $\alpha(u)$ in $\Pi^{-1}(\Pi(\alpha(u)))$, 
is orientation preserving or 
orientation reversing. This just means that 
the degree of the line bundle
$\alpha^*\nu^*$ on $B$ equals the number of 
intersection points of $\alpha(B)$ with $\Delta_E$ counted with signs (and
multiplicities- by the assumption on $\alpha$, 
these multiplicities are all one). In other words,
the degree of the line bundle $\alpha^*\nu^*$ on $B$ 
equals the intersection
number $\alpha(B)\cdot \Delta_E$. As this holds true for all 
smooth maps from a closed oriented surface into $E$, this yields
that indeed, $\Delta_E$ is Poincar\'e dual to $c_1(\nu^*)$.
\end{proof}

\section{Mumford Morita Miller classes
  are bounded}\label{mmm}

The goal of this section is to 
give a different perspective on the subvariety $\Delta\subset P^*{\cal C}$
which can be used to compute the 
Mumford Morita Miller
classes. This leads to the proof of Theorem \ref{main0}.

Consider the
$2g-2$-fold product ${\cal C}\times \dots \times {\cal C}$ of the
universal curve $\Upsilon:{\cal C}\to {\cal M}_g$. This is a smooth
complex orbifold.
The \emph{$2g-2$-fold fiber product}
${\cal Z}$ of ${\cal C}$ is defined as
\[{\cal Z}=\{(z_1,\dots,z_{2g-2})\in {\cal C}\times \cdots \times {\cal C}\mid
  \Upsilon(z_i)=\Upsilon(z_j) \text{ for all }i,j\}.\]
Then ${\cal Z}$ is a smooth complex orbifold. The next lemma
is meant in the orbifold sense, and it is immediate
from the definitions. For its formulation, let $\psi_i:{\cal Z}\to 
{\cal C}$ be the projection onto the $i$-th factor.

\begin{lemma}\label{fiberwise}
The fiberwise tangent bundle of the fibration ${\cal Z}\to {\cal M}_g$
is the holomorphic vector bundle ${\cal V}=\oplus_i \psi_i^*\nu
\to {\cal Z}$ of rank
$2g-2$, with structure group $(\mathbb{C}^*)^{2g-2}$. 
\end{lemma}

The symmetric group $\mathfrak{S}_{2g-2}$ in $2g-2$  variables acts on 
the fibers of ${\cal Z}$ as a group of 
biholomorphic transformations permuting the factors, and
this action defines a fiber preserving action on 
${\cal Z}$ by biholomorphic transformations. This action is 
however not free. The quotient 
\[{\cal D}={\cal Z}/\mathfrak{S}_{2g-2} \] of ${\cal Z}$
by this action will be called 
the \emph{bundle of effective divisors of degree 
$2g-2$}. The space ${\cal D}$ is a smooth complex orbifold
and a fiber bundle over ${\cal M}_g$ whose
fiber is a compact complex smooth orbifold and hence a
singular complex variety. The singularities of ${\cal D}$ are 
precisely the orbifold points. This means  that a generic fiber of ${\cal D}$
intersects the singular locus of ${\cal D}$ in the projection 
of the fixed point set of $\mathfrak{S}_{2g-2}$ in the
fiber of ${\cal Z}$ over the same point.

Using the previous notations, the variety ${\cal D}$ is a stratified complex
space. More precisely, we have
\[{\cal D}=\cup_k{\cal D}(k)\]
where ${\cal D}(k)$ is the 
quotient of the $\mathfrak{S}_{2g-2}$-invariant
subspace of ${\cal Z}$ of points cut out by $k$ independent equations
of the form $z_i=z_j$ $(i\not=j)$. Here $z_i$ $(1\leq i\leq 2g-2)$ is a 
point in $i$-th factor of the fiber of ${\cal Z}$. 
The open stratum is the projection of the set of points in ${\cal Z}$ with
pairwise distinct fiber components,
and the stratum ${\cal D}(1)\subset {\cal D}$ 
of codimension one is the projection 
of the set of all points in ${\cal Z}$ 
such that precisely two of the fiber components coincide.

A point in the fiber ${\cal D}_X$ 
of the bundle ${\cal D}$ over a Riemann surface $X\in {\cal M}_g$ 
can be thought of as an effective divisor of degree $2g-2$ on $X$. 
Let ${\cal D}_X^\prime\subset {\cal D}_X$ be the subset of 
${\cal D}_X$ of divisors dual to the cotangent bundle of $X$.
Then ${\cal D}^\prime=\cup_X{\cal D}_X\subset {\cal D}$ is a closed subset of ${\cal D}$.
We have

\begin{lemma}\label{smooth}
${\cal D}^\prime$ is a complex subvariety of ${\cal D}$. The map 
which 
associates to a projective differential $q\in {\cal P}$ the divisor defined by $q$
is a holomorphic bijection ${\cal P}\to {\cal D}^\prime$.
\end{lemma}
\begin{proof} 
Given an effective
divisor of degree $2g-2$ on $X$, there exists
a unique holomorphic line bundle on $X$ of degree $2g-2$ 
which admits a holomorphic section vanishing to prescribed order on the divisor. If
the divisor is dual to the cotangent bundle of $X$, then this section is a holomorphic one-form,
uniquely determined by its zero set
up to a constant multiple in  $\mathbb{C}^*$, that is,
the projection of this one-form into
${\cal P}$ does not depend on choices. This shows that 
the map $\alpha:{\cal P}\to {\cal D}^\prime$ 
which associates to a projective abelian differential $q$ on a 
Riemann surface $X$ the divisor of the
zeros of $q$ is a bijection. 

The map $\alpha$ is also easily seen to be 
continuous, with continuous inverse, and hence it is a 
homeomorphism ${\cal P}\to {\cal D}^\prime$.

We have to compare the complex structures on ${\cal P}$ and on
${\cal D}^\prime$. To this end note from Proposition \ref{localcomplete}
that the set $\Delta\subset P^*{\cal C}$ 
whose intersection with a fiber of $P^*{\cal C}$
over a projective abelian differential $q\in {\cal P}$ 
is just the set of zeros of $q$ is a complex subvariety of
$P^*{\cal C}$. Furthermore, the proof of Proposition \ref{localcomplete}
shows that the restriction of the projection
$\Pi:P^*{\cal C}\to {\cal P}$ to the open dense suborbifold
$\Delta\cap \Pi^{-1}({\cal P}(0))$ is a holomorphic
unbranched covering of degree $2g-2$.

By the definition of the complex structure on ${\cal D}$ as a quotient of the 
complex structure on ${\cal Z}$ as a holomorphic fiber bundle,
with fiber a product of $2g-2$ complex curves, 
this implies that the restriction of the map $\alpha$ to the
open stratum ${\cal P}(0)\subset {\cal P}$ 
is a biholomorphism onto its image. Thus 
the globally continuous map $\alpha:{\cal P}\to {\cal D}$
is holomorphic on the complement of a subvariety of codimension
one (the variety $\cup_{k\geq 1}{\cal P}(k)$). 
The Riemann extension theorem then shows that $\alpha$ is holomorphic on
all of ${\cal P}$.
\end{proof}

Since ${\cal D}$ is the quotient of ${\cal Z}$ under the action of the 
finite group $\mathfrak{S}_{2g-2}$ 
of biholomorphic automorphisms, the rational cohomology of ${\cal Z}$ and of 
${\cal D}$ are related by the pull-back homomorphism. 
For the formulation of the following well known result,  we denote by 
$H^*({\cal Z},\mathbb{Q})^{{\mathfrak S}_{2g-2}}$ the 
$\mathfrak{S}_{2g-2}$-invariant rational cohomology
of ${\cal Z}$.

\begin{lemma}\label{invariant}
$H^*({\cal D};\mathbb{Q})\cong H^*({\cal Z},\mathbb{Q})^{\mathfrak{S}_{2g-2}}$.
\end{lemma}
\begin{proof}
The projection ${\cal Z}\to {\cal D}$ is a finite branched cover in the sense
of Definition \ref{finitebranched} (see \cite{Ed76} for a discussion of
quotient spaces by actions of finite groups)
and hence Theorem \ref{edmonds} shows that  
there exists a transfer map $H^*({\cal Z},\mathbb{Q})\to 
H^*({\cal D},\mathbb{Q})$ whose restriction to the invariant cohomology is an isomorphism.
This implies the statement of the lemma.
We refer once more to \cite{Ed76} for more information.
\end{proof}

Our next goal is to use the results from Section \ref{branchedmultisection}
to relate the cohomology of ${\cal D}$ to tautological classes. 
As before, we work with homology classes defined by smooth maps
from closed oriented smooth manifolds, and we think of these manifolds as 
equipped with a triangulation which defines the fundamental cycle.
As we are only
interested in rational homology and cohomology, by a well known result of
Thom, this
is sufficient for our purpose. 

We begin with an elementary relation between push-forward and pull-back
of cycles and cocycles under continuous maps. 
Let  $M$ be a closed oriented manifold of dimension $n$, equipped with a 
triangulation $T$. 
Then $M$ can be viewed as a cycle of dimension $n$, given by
a standard representation of each simplex, with compatible boundary maps. 
Furthermore, if $X$ is any topological space, and if $\phi:M\to X$ is a continuous map,
then the simplicial structure on $M$ defines an image chain $\phi(M)$ which is in fact 
a cycle. As cycles form a group, they can be added. Thus for two perhaps distinct
continuous maps $\phi_1,\phi_2:M\to X$, we can consider the formal sum $\phi_1(M)+\phi_2(M)$,
viewed as a simplicial cycle which defines 
the homology class $\phi_1[M]+\phi_2[M]$. Here and in
the sequel,
we denote by $[\beta]$ the homology class of a cycle $\beta$,  and
we denote 
by $\phi[M]$ the image of the homology class $[M]$ under $\phi$
(that is, we omit the usual lower star in this notation to reduce the number
of indices). 

Recall the notation $\psi_i:{\cal Z}\to {\cal C}$ for the projection onto the $i$-th factor in 
the fiber product ${\cal Z}$.

\begin{lemma}\label{evaluate2}
Let $M$ be a closed oriented manifold of dimension $2n$ and let 
$\phi:M\to {\cal Z}$ be a continuous map. Then $\phi(M)$ defines a simplicial cycle
$\beta(\phi)=\sum_i\psi_i\circ \phi(M)\subset {\cal C}$ 
with the property that 
\[ \frac{(-1)^n}{n!}c_1(\nu^*)^n[\beta(\phi)]={\rm ch}({\cal V})(\phi[M]).\] 
Furthermore, if $\xi\in \mathfrak{S}_{2g-2}$ is arbitrary, then we have 
$\beta(\xi\circ \phi)=\beta(\phi)$. 
\end{lemma}
\begin{proof}
  Since 
  ${\cal V}=\oplus_i \psi_i^*\nu$ by Lemma \ref{fiberwise} 
  and since for complex vector bundles
$V,W$ over the same topological space it holds that 
\[{\rm ch}(V\oplus W)={\rm ch}(V)+{\rm ch}(W)\]
we have 
\begin{align}
{\rm ch}({\cal V})(\phi[M])&=\sum_i {\rm ch}(\psi_i^*\nu)(\phi[M])\notag \\=
\frac{1}{n!}
\sum_ic_1(\psi_i^*\nu)^n(\phi[M]) &= \frac{1}{n!}
\sum_ic_1(\nu)^n(\psi_i\phi[M])=
\frac{(-1)^n}{n!}c_1(\nu^*)^n[\beta(\phi)]\notag
\end{align}
as claimed in the lemma.

Invariance of $\beta(\phi)$ 
under postcomposition of $\phi$ with an element of $\mathfrak{S}_{2g-2}$ is immediate
from the construction. 
\end{proof}

Let $\alpha:{\cal P}\to {\cal D}^\prime$ be the holomorphic homeomorphism 
defined in Lemma \ref{smooth}. Recall that ${\cal Z}\to {\cal D}$ is a branched
covering of degree $\vert \mathfrak{S}_{2g-2}\vert =(2g-2)!$. Thus there is a
transfer map $H_*({\cal D},\mathbb{Q})\to H_*({\cal Z},\mathbb{Q})$. 
This transfer map maps a class defined by a simplicial cycle which is compatible with the stratification 
${\cal D}=\cup_k{\cal D}(k)$ to the class of its full preimage in ${\cal Z}$, which is 
a $\mathfrak{S}_{2g-2}$-invariant cycle. Recall from Section \ref{branchedmultisection} the definition of the
projection $\Xi:{\cal H}_*\to {\cal P}$. 

\begin{proposition}\label{verticalshriek}
Let $M$ be a closed oriented manifold of dimension $2n$ and let
$F:M\to {\cal S}$ be a smooth map which is transverse to the stratification
${\cal S}=\cup_k{\cal S}(k)$. Let $\mu$ be the homology class of the preimage of 
$\alpha\circ  \Xi\circ F(M)\subset {\cal D}^\prime$ in ${\cal Z}$; then 
\[\frac{-n!}{(2g-2)!}{\rm ch}({\cal V})(\mu)=\kappa_n(P\circ \Xi\circ F[M]).\]
\end{proposition}  
\begin{proof} By Theorem \ref{branchedpoincare}, there exists a simplicial complex
$\hat M$ of homogeneous dimension $2n$, and 
there exists a branched multisection $\sigma:\hat M\to E$
of the surface bundle $E=(P\circ \Xi \circ F)^*{\cal C}\to M$ with classifying map 
$P\circ \Xi \circ F:M\to {\cal M}_g$ 
such that 
\begin{align}
\kappa_n(P\circ \Xi\circ F[M]) &=c_1(\nu)^{n+1}[E]=(-1)^{n+1}c_1(\nu^*)^{n+1}[E]\notag\\
& =
(-1)^{n+1}c_1(\nu^*)^n(\sigma[\hat M]).\notag\end{align}
Furthermore, the image of $\sigma(\hat M)$
in $E$ is just the pullback $(\Xi \circ F)^*(\Delta)$ of the subvariety $\Delta\subset P^*{\cal C}$ 
constructed in Section \ref{acomplex}.

Since $F$ is transverse to the stratification ${\cal S}=\cup_k{\cal S}(k)$, we can choose a triangulation
$T$ of $M$ compatible with the induced stratification of $M$.
The image of $M$ under the map $\alpha\circ \Xi \circ F:M\to {\cal D}^\prime$, viewed as a map
from a simplicial complex, is a 
simplicial cycle in ${\cal D}^\prime\subset {\cal D}$ which is compatible
with the stratification. This means that for each $k\geq 0$, its intersection with the singular locus 
$\cup_{j\geq k}{\cal D}(j)$ is a subcomplex of codimension $2k$. As ${\cal D}(k)$ is just the subset
of the branch locus of the branched covering ${\cal Z}\to {\cal D}$ of all points with 
precisely $\frac{(2g-2)!}{(k+1)!}$ preimages, the image in ${\cal D}^\prime$ of 
each open simplex of $T$ of dimension $2n$ or $2n-1$ has
precisely $(2g-2)!$ preimages in ${\cal Z}$. Thus by compatibility with boundary maps, the 
preimage of this cycle in ${\cal Z}$ is a $\mathfrak{S}_{2g-2}$-invariant
cycle $\chi$ in ${\cal Z}$ whose 
homology class $[\chi]$ is mapped by the projection
$H_*({\cal Z},\mathbb{Q})\to H_*({\cal D},\mathbb{Q})$ to $(2g-2)! [\alpha\circ \Xi\circ F(M)]$
as specified in Theorem \ref{edmonds}.

By Lemma \ref{evaluate2}, 
the evaluation of $\frac{(-1)^n}{n!}c_1(\nu^*)^n$ on $\sum_i\psi_i[\chi]\in H_{2n}({\cal C},\mathbb{Q})$
equals 
${\rm ch}({\cal V})[\chi]$. However, as these constructions all preserve the fibers of the various bundles over
${\cal M}_g$,  the cycle $\sum_i\psi_i(\chi)$ is contained in the restriction of ${\cal C}$ to 
$P\circ \Xi\circ F(M)$ and can be viewed as a cycle in the pull-back surface bundle $E$.
Furthermore, with this interpretation, it 
consists of $(2g-2)!$ copies of $(\Xi\circ F)^*(\Delta)$. 
As a consequence of Proposition \ref{poincaredual}, we have
\begin{align}
(-1)^nn! {\rm ch}({\cal V}) &= c_1(\nu^*)^n(\sum_i\psi_i[\chi]) =(2g-2)!c_1(\nu^*)^n((\Xi\circ F)^*(\Delta))\notag\\
&=(2g-2)!c_1(\nu^*)^{n}(\sigma[\hat M])=
(-1)^{n+1}(2g-2)!\kappa_n(P\circ \Xi\circ F(M)).\notag\end{align} 
Together this yields the proposition.
\end{proof}

%
%




The group $\mathfrak{S}_{2g-2}$ acts by permutations as a group of  
outer automorphisms on $(\mathbb{C}^*)^{2g-2}$.
We write this action from the left. 
Thus we can form the semi-direct product
$\mathfrak{S}_{2g-2}\ltimes (\mathbb{C}^*)^{2g-2}$. The following short lemma
is not needed in the sequel but may be useful as an illustration of our
constructions.

\begin{lemma}\label{embed}
  The semi-direct product $\mathfrak{S}_{2g-2}\ltimes (\mathbb{C}^*)^{2g-2}$
  can be 
realized as an extension of a subgroup of 
$GL(2g-2,\mathbb{C})$ by a group of inner automorphisms
of $GL(2g-2,\mathbb{C})$.
\end{lemma}
\begin{proof}
The subgroup of $GL(2g-2,\mathbb{C})$ of diagonal matrices 
is naturally isomorphic to $(\mathbb{C}^*)^{2g-2}$. On the other hand,
the matrix
\[ J=\begin{pmatrix} 0 & 1 &\\
1 & 0 & \\
 &  & {\rm Id}
\end{pmatrix}\]
squares to the identity, and if 
$A\in GL(2g-2,\mathbb{C})$ is arbitrary, then the product $JA$ is obtained from $A$ by
exchanging the first and the second row, while $AJ$ is obtained from
$A$ by exchanging the first and second column. As a consequence, 
conjugation by $J$ preserves the diagonal matrices and transposes the
first and second diagonal entry. 
 
More generally, for any  $1\leq i<j\leq 2g-2$ there is a 
matrix in $GL(2g-2,\mathbb{C})$ of a similar form which 
squares to the identity and whose action by conjugation 
on the diagonal matrices 
exchanges the $i$-th and $j$-the diagonal entry. Thus indeed,
the semi-direct product $\mathfrak{S}_{2g-2}\ltimes 
(\mathbb{C}^*)^{2g-2}$ is an extension of $(\mathbb{C}^*)^{2g-2}$
by a group of inner automorphisms of $GL(2g-2,\mathbb{C})$. 
\end{proof}

Let now $E \mathfrak{S}_{2g-2}$ be a classifying space 
for $\mathfrak{S}_{2g-2}$, that is, a contractible cell complex on which 
$\mathfrak{S}_{2g-2}$ acts freely. 
Then we can form the quotient
\[{\cal X}={\cal Z}\times E\mathfrak{S}_{2g-2}/\mathfrak{S}_{2g-2}\] 
which is a rational
cohomology model for ${\cal D}$. Namely, ${\cal Z}\times E\mathfrak{S}_{2g-2}$ is 
homotopy equivalent to ${\cal Z}$, and the $\mathfrak{S}_{2g-2}$-invariant rational 
cohomology for the diagonal action can be identified with 
the invariant rational cohomology for the action of $\mathfrak{S}_{2g-2}$
on ${\cal Z}$. 

Let
$\theta:H^*({\cal D},\mathbb{Q})\to H^*({\cal Z},\mathbb{Q})$ be the pull-back homomorphism, and 
let $\tau$ be the composition of the pull-back homomorphism 
$H^*({\cal Z},\mathbb{Q})\to H^*({\cal Z}\times E\mathfrak{S}_{2g-2},\mathbb{Q})$ 
with the transfer map $H^*({\cal Z}\times E\mathfrak{S}_{2g-2},\mathbb{Q})\to 
H^*({\cal Z}\times \mathfrak{S}_{2g-2}/\mathfrak{S}_{2g-2},\mathbb{Q})$. 
Then 
the sequence 
\[H^*({\cal D},\mathbb{Q})\xrightarrow{\theta} H^*({\cal Z},\mathbb{Q})\xrightarrow{\tau}
H^*({\cal Z}\times E\mathfrak{S}_{2g-2}/\mathfrak{S}_{2g-2},\mathbb{Q})\]
consists of isomorphisms, 
where $H^*({\cal Z}\times E\mathfrak{S}_{2g-2}/\mathfrak{S}_{2g-2},\mathbb{Q})$
is usually called the equivariant rational cohomology of ${\cal Z}$.  

The differentials of 
the elements of $\mathfrak{S}_{2g-2}$, viewed as a group of
biholomorphic automorphisms of ${\cal Z}$, act as a group of 
biholomorphic bundle
automorphisms on the vertical tangent bundle ${\cal V}$ of ${\cal Z}$. 
As a consequence, the Chern classes of this bundle are $\mathfrak{S}_{2g-2}$-invariant.

To investigate these Chern classes let $\hat {\cal V}$  be the pull-back of 
${\cal V}$ to ${\cal Z}\times E\mathfrak{S}_{2g-2}$. 
We have

\begin{lemma}\label{borel}
There is a rank $2g-2$ complex vector bundle
${\cal W} \to {\cal X}$ with structure group
$\mathfrak{S}_{2g-2}\ltimes (\mathbb{C}^*)^{2g-2}$
which pulls back to $\hat {\cal V}\to {\cal Z}\times E\mathfrak{S}_{2g-2}$.
\end{lemma}
\begin{proof}
  Since $\hat {\cal V}$ is the pull-back of the bundle ${\cal V}$ 
  on ${\cal Z}$ and $\mathfrak{S}_{2g-2}$ acts on each fiber
  of ${\cal Z}$ as a group of biholomorphic transformations, 
  it induces an action as a group of bundle automorphisms of
  $\hat {\cal V}$. As the action of $\mathfrak{S}_{2g-2}$ 
  on ${\cal Z}\times E\mathfrak{S}_{2g-2}$ is free,
  the same holds true for the action of
  $\mathfrak{S}_{2g-2}$ on $\hat {\cal V}$. Thus we can form the
  quotient ${\cal W}=\hat {\cal V}/\mathfrak{S}_{2g-2}$, and this
  quotient has naturally the structure of a vector bundle
  over ${\cal Z}\times E\mathfrak{S}_{2g-2}/\mathfrak{S}_{2g-2}
  ={\cal X}$.
  By construction, the pull-back of ${\cal W}$ to ${\cal Z}\times
  E\mathfrak{S}_{2g-2}$ is just the bundle $\hat {\cal V}$.

  Since the structure group of the bundle $\hat {\cal V}$ is the group
$(\mathbb{C}^*)^{2g-2}$, 
the structure group of the bundle 
${\cal W}$ is the semi-direct product
$\mathfrak{S}_{2g-2}\ltimes (\mathbb{C}^*)^{2g-2}$. 
\end{proof}

Since the group $\mathfrak{S}_{2g-2}$ acts properly on ${\cal Z}$, with 
quotient ${\cal D}$, and it 
acts freely on $E\mathfrak{S}_{2g-2}$, the first factor projection
${\cal Z}\times E\mathfrak{S}_{2g-2}\to {\cal Z}$ induces a projection 
$\Theta:{\cal X}\to {\cal D}$. This projection 
induces an isomorphism $H^*({\cal D},\mathbb{Q})\to H^*({\cal X},\mathbb{Q})$
as explained before Lemma \ref{borel}.

The Chern character
of ${\cal W}$ 
is a formal sum of cohomology classes in the cohomology ring
$H^{*}({\cal X},\mathbb{Q})=H^{*}({\cal D},\mathbb{Q})$.
As the Chern character of a vector bundle is equivariant with respect to 
pull-back, we obtain

\begin{corollary}\label{interprete}
  The Chern character ${\rm ch}({\cal W})\in H^*({\cal X},\mathbb{Q})$
  of ${\cal W}$ pulls back to the Chern character of  $\hat {\cal V}$.
  Moreover, it can be viewed as   
a sum of cohomology classes in $H^*({\cal D},\mathbb{C})$.
\end{corollary}

Denote by ${\rm ch}({\cal W})_k$
the component of ${\rm ch}({\cal W})$ in $H^{2k}({\cal X},\mathbb{Q})$.

\begin{proposition}\label{invariant3}
  The pull-back to ${\cal S}\subset {\cal H}$
  of the $n$-th Mumford Morita Miller 
  class $\kappa_n\in H^{2n}({\cal M}_g,\mathbb{Q})$ by the canonical
  projection ${\cal S}\to {\cal M}_g$ 
coincides with the pull-back of the class 
$-n! {\rm ch}({\cal W})_n
\in H^{2n}({\cal D},\mathbb{Q})$ by the map 
$\alpha\circ \Xi:{\cal S}\to {\cal D}^\prime\subset {\cal D}$.
\end{proposition}
\begin{proof} Let $M$ be a closed oriented manifold of dimension
$2n\in \{2,\dots,2g-2\}$,
let $F:M\to {\cal S}$ be a smooth map  
transverse to the stratification ${\cal S}=\cup_k{\cal S}(k)$ and define
$\zeta=\alpha\circ \Xi\circ F:M\to {\cal D}^\prime$.
Then $\zeta(M)$ is a cycle in ${\cal D}^\prime\subset {\cal D}$. 
As homology with rational coefficients can be generated by
maps from manifolds, all we have to verify is that
$-n!{\rm ch}({\cal W})_n(\zeta[M])=\kappa_n(P\circ \Xi\circ F[M])$.

Let $\tilde \zeta[M]\in H_{2k}({\cal Z},\mathbb{Q})$ be the image of 
$\zeta[M]$ under the transfer map, defined by the projection 
${\cal Z}\to {\cal D}$, 
and let $\hat \zeta[M]$ be the 
image of $\tilde \zeta[M]$ under the 
identification $H_*({\cal Z},\mathbb{Q})\cong 
H_*({\cal Z}\times E\mathfrak{S}_{2g-2},\mathbb{Q})$ by homotopy equivalence.
Using the identification of 
invariant and equivariant cohomology and naturality under pull-back, we know that 
\[{\rm ch}({\cal V})(\tilde \zeta[M])={\rm ch}(\hat {\cal V})(\hat \zeta[M])=
(2g-2)!{\rm ch}({\cal W})(\zeta[M]).\] 
From
Proposition \ref{verticalshriek} we then obtain that  
\[\kappa_n(P\circ \Xi\circ F[M])=\frac{-n!}{(2g-2)!}{\rm ch}({\cal V})(\tilde \zeta[M])=
-n! {\rm ch}({\cal W})(\zeta[M])\] which is what we wanted to show.
\end{proof}

\begin{corollary}\label{bounded}
The classes $\kappa_n\in H^{2n}({\cal M}_g,\mathbb{Q})$ are bounded.
\end{corollary}
\begin{proof}
By Theorem 4-3 of \cite{M88}, the class $c_1(\nu)\in H^2({\cal C},\mathbb{Q})$ is bounded.
As a consequence, the pull-backs 
$\psi_i^*c_1(\nu)$ of this class to 
${\cal Z}$ are bounded as well. Then the same holds true for 
${\rm ch}({\cal V})$ whose components are polynomials in the classes 
$\psi_i^*c_1(\nu)$. 

The Chern character ${\rm ch}({\cal V})$ is $\mathfrak{S}_{2g-2}$-invariant and 
descends to ${\rm ch}({\cal W})$. Thus ${\rm ch}({\cal W})\in H^*({\cal D},\mathbb{Q})$ 
is bounded, and hence the same holds true for the image of ${\rm ch}({\cal W})$
under the homomorphism $H^*({\cal D},\mathbb{Q})\to H^*({\cal D}^\prime,\mathbb{Q})$ 
induced by the inclusion. 
Now ${\cal D}^\prime$ is homeomorphic to ${\cal P}$ and therefore
${\rm ch}({\cal W})\in H^*({\cal P},\mathbb{Q})$ is bounded. 
The same then holds true for its image under the homomorphism
$H^*({\cal P},\mathbb{Q})\to H^*({\cal S},\mathbb{Q})$ induced by the projection
${\cal S}\to {\cal P}$.  

Since we have 
\[H^*({\cal P},\mathbb{Q})=H^*({\cal M}_g,\mathbb{Q})[\eta]/(\eta^{g}+c_1({\cal H})\eta^{g-1}
+\cdots +c_{g}({\cal H}))\]
where $\eta$ denotes the Chern class of the tautological bundle over the fibers of 
${\cal P}\to {\cal M}_g$ and where 
$c_i({\cal H})$ is a polynomial in the classes $\kappa_i$ \cite{M87},  
the pull-back of $H^*({\cal P},\mathbb{Q})$ to ${\cal S}$ 
coincides with the pull-back to ${\cal S}$ of the cohomology ring $H^*({\cal M}_g,\mathbb{Q})$. 
Thus by Proposition \ref{invariant3}, the component in $H^{2n}({\cal S},\mathbb{Q})$ of the
pull-back of ${\rm ch}({\cal W})$ to ${\cal S}$ coincides up to a nonzero factor with
the pull-back of the class $\kappa_n$ for $n\leq g-1$. 
From the fact that $\kappa_n=0$ 
for $n\geq g-1$ \cite{Lo95} we conclude that indeed, all Mumford Morita
Miller classes are bounded. 
\end{proof}

\section{A cocycle for $\kappa_1$}\label{cocycle}

In this final section we apply the main result of Section \ref{mmm} 
to complete the proof of 
Theorem \ref{main}. We use the assumption and notations from the previous sections.

From now on we consider a surface bundle $E\to B$ over a closed oriented surface $B$.
Recall from Section \ref{mmm}
the definition of the complex vector bundle
\[{\cal W}\to {\cal Z}\times E\mathfrak{S}_{2g-2}/\mathfrak{S}_{2g-2}={\cal X}.\]
Let $c_1({\cal W})$ be the first Chern class of ${\cal W}$. By the discussion in 
Section \ref{mmm}, we view $c_1({\cal W})$ as a cohomology class in 
$H^2({\cal D},\mathbb{Q})$ where as before, ${\cal D}\to {\cal M}_g$ is the bundle of 
divisors over ${\cal M}_g$. 
If $\zeta:B\to {\cal D}$ is any smooth map   
then the pull-back $\zeta^*(c_1({\cal W}))$ 
to $B$ is defined.

The following proposition is a reformulation of Proposition \ref{invariant3}.
We denote as before by $\sigma(E)$ the signature
of a surface bundle over a surface $E\to B$. 
Recall from Lemma \ref{lift} that up to homotopy, a smooth classifying map $f:B\to {\cal M}_g$ for $E$
admits a smooth lift $F:B\to {\cal S}$, transverse to the stratification of ${\cal S}$, 
and the map $\Xi\circ F:B\to {\cal P}$ determines via the homeomorphism
${\cal P}\to {\cal D}^\prime\subset {\cal D}$ a map $\zeta:B\to {\cal D}^\prime$
which we call \emph{characteristic} for $E$. Note that the homotopy class
of a characteristic map only depends on the homotopy class of $f$ and hence only
depends on $E$.

\begin{theorem}\label{identify}
Let $\Pi:E\to B$ be a surface bundle over a surface 
and let $\zeta:B\to {\cal D}$ be a characteristic map for $E$;
then 
\[3\sigma(E)=-\zeta^*(c_1({\cal W}))[B].\]
\end{theorem} 
\begin{proof}
It suffices to observe that 
\[3\sigma(E)=f^*\kappa_1[B]=
-\zeta^*{\rm ch}({\cal W})[B]=-\zeta^*c_1({\cal W})[B]\]
by Proposition \ref{invariant3}.  
\end{proof}

The \emph{norm} $\Vert \alpha \Vert_\infty$ 
of a bounded cohomology class
$\alpha$ is the infimum of the $L^\infty$-norms of the cocycles
representing $\alpha$. We know that the 
first Mumford Morita Miller 
class $\kappa_1$ is bounded
(\cite{M87} and Corollary \ref{bounded}), thus its norm is
finite. Our main goal is to show that $\Vert \kappa_1\Vert_\infty \leq g-1$ 
by constructing an explicit cocycle for $\kappa_1$ 
with values in $[-g+1,g-1]$. From this estimate, Theorem \ref{main}
follows from a standard argument which is reproduced in
Corollary \ref{milnorwood} below.

Consider the group ${\rm Top}^+(S^1)$ of orientation preserving homeomorphisms
of the circle $S^1$. The direct product $({\rm Top}^+(S^1))^{2g-2}$ 
of $2g-2$ copies of ${\rm Top}^+(S^1)$ is a subgroup of the 
group of homeomorphisms of the torus $(S^1)^{2g-2}$.

Let $\widetilde{\rm Top}^+(S^1)$ be the group of orientation preserving 
homeomorphisms of the real line that commute with integral translations.
This is just the universal covering of ${\rm Top}^+(S^1)$. It gives rise to a 
central extension
\[0\to \mathbb{Z}\to \widetilde{\rm Top}^+(S^1)\to {\rm Top}^+(S^1)\to 0\] 
(see p.113 of \cite{Fr17} for details). Taking a $2g-2$-fold product of this 
sequence defines a central extension
\begin{equation}\label{extension}
0\to \mathbb{Z}^{2g-2}\to (\widetilde{\rm Top}^+(S^1))^{2g-2}\to 
({\rm Top}^+(S^1))^{2g-2}\to 0.\end{equation}

The symmetric group $\mathfrak{S}_{2g-2}$ acts on $(S^1)^{2g-2}$
and on $\mathbb{R}^{2g-2}$ as
a group of permutations  
of the factors, and this action 
normalizes the groups $({\rm Top}^+(S^1))^{2g-2}$
and $(\widetilde{\rm Top}^+(S^1))^{2g-2}$.
In other words,
$\mathfrak{S}_{2g-2}$ acts on 
$({\rm Top}^+(S^1))^{2g-2}$ and on $(\widetilde{\rm Top}^+(S^1))^{2g-2}$
as an outer automorphism group, defined by 
permutations of the factors. This action commutes with the projection 
$(\widetilde{\rm Top}^+(S^1))^{2g-2}\to ({\rm Top}^+(S^1))^{2g-2}$. 
By naturality, we therefore obtain from the short exact
sequence (\ref{extension})
an exact sequence of group extension of the form
\begin{equation}\label{exactsequence}
0\to \mathbb{Z}^{2g-2} 
\xrightarrow{\iota}
\mathfrak{S}_{2g-2}\ltimes (\widetilde{\rm Top}^+(S^1))^{2g-2}\xrightarrow{\theta}
\mathfrak{S}_{2g-2}\ltimes ({\rm Top}^+(S^1))^{2g-2}\to 0.
\end{equation}

Let $\rho:\mathbb{Z}^{2g-2}\to \mathbb{Z}$ be defined by
\begin{equation}\label{rho}
\rho(x_1,\dots,x_{2g-2})=\sum_i x_i. \end{equation}
Then $\rho$ is invariant under the action of 
$\mathfrak{S}_{2g-2}$ by permutations of the standard generators
determined by the product structure which was used in the construction of the sequence
(\ref{extension}). As a consequence, the  
inclusion $\iota:\mathbb{Z}^{2g-2}\to (\widetilde{\rm Top}^+(S^1))^{2g-2}$ 
maps the kernel of $\rho$ to a $\mathfrak{S}_{2g-2}$-invariant subgroup of 
$(\widetilde{\rm Top}^+(S^1))^{2g-2}$.
Taking the quotient of the first two groups in 
the sequence (\ref{exactsequence}) by the kernel of $\rho$ and its image under
the inclusion $\iota:\mathbb{Z}^{2g-2}\to 
\mathfrak{S}_{2g-2}\ltimes (\widetilde{\rm Top}^+(S^1))^{2g-2}$, respectively, 
yields an infinite cyclic extension
\begin{equation}\label{cyclicex}
0\to \mathbb{Z}\to G\xrightarrow{\chi} \mathfrak{S}_{2g-2}\ltimes ({\rm Top}^+(S^1))^{2g-2} \to 0\end{equation}
where $G=\mathfrak{S}_{2g-2}\ltimes(\widetilde{\rm Top}^+(S^1))^{2g-2}/ \iota({\rm ker}(\rho))$.

Since the extension (\ref{extension}) is central and the homomorphism 
$\rho$ commutes with the action of the group $\mathfrak{S}_{2g-2}$, the 
extension (\ref{cyclicex}) is central and hence it defines a class
\[e\in H^2(\mathfrak{S}_{2g-2}\ltimes ({\rm Top}^+(S^1))^{2g-2}, \mathbb{Q}).\]
We refer to \cite{Br82} for details of the relation between cyclic central extensions
of a group and second cohomology, see also 
\cite{Fr17} in the context of bounded cohomology.

The following is the key observation towards a proof of
Theorem \ref{main}.

\begin{proposition}\label{controlextension}
  The cohomology class $e\in H^2(\mathfrak{S}_{2g-2}\ltimes
  ({\rm Top}^+(S^1))^{2g-2},\mathbb{Q})$ can be represented by a cocycle
  which takes values in $\{-g+1,\dots,g-1 \}$.
\end{proposition}  
\begin{proof}
  The construction follows the construction of
  a cocycle for the bounded Euler
  class of a flat circle bundle as described in Chapter 10 of \cite{Fr17}.

Fix a point $x_0\in \mathbb{R}$.
For $h\in {\rm Top}^+(S^1)$ let $\tilde h_{x_0}\in \widetilde{\rm Top}^+(S^1)$ 
be the unique lift of $h$ so that
$\tilde h_{x_0}(x_0)-x_0\in [0,1)$. These lifts then define a set
theoretic section 
\[s_{x_0}:\mathfrak{S}_{2g-2}\ltimes ({\rm Top}^+(S^1))^{2g-2}\to 
\mathfrak{S}_{2g-2}\ltimes (\widetilde{\rm Top}^+(S^1))^{2g-2}\]
of the projection homomorphism 
\[\theta:\mathfrak{S}_{2g-2}\ltimes (\widetilde{\rm Top}^+(S^1))^{2g-2}\to 
\mathfrak{S}_{2g-2}\ltimes ({\rm Top}^+(S^1))^{2g-2},\]
that is, a map $s_{x_0}$ which satisfies 
$\theta \circ s_{x_0}={\rm Id}$,  
as follows.

The semi-direct product 
$\mathfrak{S}_{2g-2}\ltimes ({\rm Top}^+(S^1))^{2g-2}$ 
is the image of the free product 
\[\mathfrak{S}_{2g-2}* ({\rm Top}^+(S^1))^{2g-2}\]
by a surjective homomorphism. Thus an element of 
$\mathfrak{S}_{2g-2}\ltimes ({\rm Top}^+(S^1))^{2g-2}$ can be represented 
by a pair $(\sigma,h)$ which consists of a 
$(2g-2)$-tuple
$h=(h^1,\dots,h^{2g-2})\in ({\rm Top}^+(S^1))^{2g-2}$ and 
a permutation $\sigma\in \mathfrak{S}_{2g-2}$. Viewed
as an element of the group of homeomorphisms of the torus $(S^1)^{2g-2}$, 
the action of $(\sigma,h)$ on $(S^1)^{2g-2}$ is given by
\begin{align}
(\sigma,h)(x^1,\dots,x^{2g-2}) &=\sigma(h^1(x^1),\dots,h^{2g-2}(x^{2g-2}))\notag\\
&=
  (h^{\sigma^{-1}(1)}(x^{\sigma^{-1}(1)}),\cdots,
  h^{\sigma^{-1}(2g-2)}(x^{\sigma^{-1}(2g-2)})).\notag
\end{align}
The group multiplication reads (written from left to right, that is, we apply first the left
map, followed by the right map)
\begin{align}
(\sigma,g)\cdot (\eta,h)(x^1,\dots,x^{2g-2})\notag\\ =
  \eta(h^1(g^{\sigma^{-1}(1)}x^{\sigma^{-1}(1)}),\dots,
  h^{2g-2}(g^{\sigma^{-1}(2g-2)}x^{\sigma^{-1}(2g-2)})).\end{align}

For $(\tilde x^1,\dots, \tilde x^{2g-2})\in \mathbb{R}^{2g-2}$ define
\[s_{x_0}(\sigma, (h^1,\dots,h^{2g-2}))(\tilde x^1,\dots,\tilde x^{2g-2})=
\sigma (\widetilde{h^1}_{x_0}(\tilde x^1),\dots,\widetilde{h^{2g-2}}_{x_0}(\tilde x^{2g-2})).\] 
The map $s_{x_0}$ is equivariant with respect to the action of 
$\mathfrak{S}_{2g-2}$, that is, we have 
\[s_{x_0}(\sigma,h)=\sigma(s_{x_0}(e,h)),\]
and it defines a set theoretic section of 
$\theta$ (just apply the definition).

Using the homomorphism $\rho:\mathbb{Z}^{2g-2}\to \mathbb{Z}$ defined in 
equation (\ref{rho}), 
for a pair $((\sigma,g),(\eta,h))\in \mathfrak{S}_{2g-2}\ltimes ({\rm Top}^+(S^1))^{2g-2}$
define 
\begin{equation}\label{cocycle2}
c_{x_0}((\sigma,g),(\eta,h))=
\rho(s_{x_0} ((\sigma,g)\cdot (\eta,h))^{-1}s_{x_0}(\sigma,g)
s_{x_0}(\eta,h))\in \mathbb{Z}\end{equation}
where we identify $ \mathbb{Z}^{2g-2}$
with its image in 
$\mathfrak{S}_{2g-2}\ltimes (\widetilde{\rm Top}^+(S^1))^{2g-2}$ by the inclusion 
$\iota$ given by the exact sequence (\ref{exactsequence}). 
That this makes sense follows from the fact that $s_{x_0}$ is a set theoretic section of
the homomorphism $\theta$ and hence 
\[s_{x_0} ((\sigma,g)\cdot (\eta,h))^{-1}s_{x_0}(\sigma,g) s_{x_0}(\eta,h)\in {\rm ker}(\theta)=
{\rm im}(\iota).\]

We claim that the cochain
$c_{x_0}$ for the group $\mathfrak{S}_{2g-2}\ltimes ({\rm Top}^+(S^1))^{2g-2}$
is a cocycle which 
represents the class $e$ defined by the central extension (\ref{cyclicex}). 
However, this is standard (see \cite{Br82,Fr17}). Namely,
given a central extension $G$ of a group $H$ by the group $\mathbb{Z}$ and
a set theoretic section $\sigma:H\to G$ which maps the identity in 
$H$ to the identity in $G$, the expression
\[c(g,h)=\sigma(gh)^{-1}\sigma(g)\sigma(h)\in \mathbb{Z}\]
defines a cocycle via the bar resolution 
where as before, we identify $\mathbb{Z}$ with its
image under the inclusion $\mathbb{Z}\to G$ defining the central extension $G$ 
(see Theorem IV.3.12 of \cite{Br82}), and this cocycle 
defines the class in $H^2(H,\mathbb{Z})$ which determines $G$. 
  
In the situation of the exact sequence (\ref{exactsequence}), we  
take the set theoretic section $s_{x_0}:\mathfrak{S}_{2g-2}\ltimes ({\rm Top}^+(S^1))^{2g-2}\to 
\mathfrak{S}_{2g-2}\ltimes(\widetilde{\rm Top}^+(S^1)^{2g-2})$ and project it into the quotient
by the normal subgroup $\iota({\rm ker}(\rho))$ to obtain a set theoretic section of the
surjection $\chi$ in the 
infinite cyclic extension (\ref{cyclicex}). This construction 
defines a cocycle with values in $\mathbb{Z}$.
As taking this quotient just amounts to applying the homomorphism $\rho$, this shows
that the cocycle (\ref{cocycle2}) 
indeed defines the cohomology class $e$.

We next bound the $L^\infty$-norm of the cocycle $c_{x_0}$. 
To this end let 
\[(\sigma,(h^1,\dots,h^{2g-2})),
(\eta,(g^1,\dots,g^{2g-2}))\in \mathfrak{S}_{2g-2}\ltimes 
({\rm Top}^+(S^1))^{2g-2},\] with lifts 
$s_{x_0}(\sigma,g)=(\sigma,(\widetilde{h^1}_{x_0},\dots,\widetilde{h^{2g-2}}_{x_0}))$ and 
$s_{x_0}(\eta,h)=(\eta,(\widetilde{g^1}_{x_0},\dots,\widetilde{g^{2g-2}}_{x_0}))$. 
The value $c_{x_0}((\sigma,g),(\eta,h))$ can be computed as follows.

Consider the permutation $\sigma \cdot \eta$ (read from left to right, that is,
we apply first $\sigma$ and then $\eta$). For each $i$
write $u^{\sigma\eta(i)}=\widetilde{g^i}_{x_0} \widetilde{h^{\sigma(i)}}_{x_0}$ (read from left to right, 
that is, we apply first $\widetilde{g^i}_{x_0}$ and then $\widetilde{h^{\sigma(i)}}_{x_0}$).
This defines a $(2g-2)$-tuple of elements
of $({\rm Top}^+(S^1))^{2g-2}$, and 
using the above description of the elements of 
$\mathfrak{S}_{2g-2}\ltimes ({\rm Top}^+(S^1))^{2g-2}$, we obtain
\[(\sigma,g)\cdot (\eta,h)=(\sigma \cdot \eta,u).\]

Apply the set theoretic section $s_{x_0}$ to obtain
\[s_{x_0}((\sigma,g)\cdot (\eta,h))^{-1}s_{x_0}(\sigma,g)s_{x_0}(\eta,h)=
(e,p_1,\dots, p_{2g-2})\in \mathfrak{S}_{2g-2}\ltimes \mathbb{Z}^{2g-2}\]
where $p_{\sigma \eta(i)}=(\widetilde{(g^ih^{\sigma(i)})}_{x_0})^{-1}
\widetilde{g^i}_{x_0}\widetilde{h^{\sigma(i)}}_{x_0}$. 
Following the computation on p.114 of \cite{Fr17}, we conclude that
$p_i\in \{0,1\}$ and hence by definition, we have
\[\rho(e,p_1,\dots,p_{2g-2})\in \{0,\dots,2g-2\}.\] 

To complete the proof of the Proposition, just
add to the cocycle $c_{x_0}$ the constant
cocycle which maps all points to $-g+1$. This cocycle is a coboundary and hence
the resulting cocycle still 
represents the class of $c_{x_0}$, but it takes values in 
$\{-g+1,\dots,g-1\}$ (see Lemma 12.1 of \cite{Fr17} for details on this construction).
\end{proof}

\begin{remark}
Another way to show Proposition \ref{controlextension} would be to work directly
in the orbifold covering ${\cal Z}$ of ${\cal D}$, that is, to work 
with the group $({\rm Top}^+(S^1))^{2g-2}$, and to observe that the resulting
cocycle is invariant under the action of the group $\mathfrak{S}_{2g-2}$. 
We chose not to adapt this strategy as we felt that 
the proof presented (although a bit more technical) 
explains better the underlying mechanism for the norm bound for $\kappa_1$.
\end{remark}

We are now ready to complete the proof of Theorem \ref{main} from the introduction.

\begin{proposition}\label{final}
$c_1({\cal W})\in H^2({\cal D},\mathbb{Q})$ 
is a bounded cohomology class, with norm $\Vert c_1({\cal W})\Vert_\infty\leq g-1$.
\end{proposition}
\begin{proof}
By \cite{M88} (see also \cite{H12} for a geometric proof), the circle subbundle of the 
vertical tangent
bundle $\nu$ of ${\cal C}\to {\cal M}_g$ is flat. By this we mean that there exists a
homomorphism $\chi:{\rm Mod}(S_{g,1})\to {\rm Top}^+(S^1)$ such that
this circle bundle can be represented as 
\[{\cal T}_{g,1}\times S^1/{\rm Mod}(S_{g,1})\]
where ${\rm Mod}(S_{g,1})$ acts on the Teichm\"uller space 
${\cal T}_{g,1}$ of marked Riemann surfaces of genus $g$ with a single marked
point by precomposition of marking, and it acts on ${\rm Top}^+(S^1)$ via
the representation $\chi$.

Consider the fiberwise tangent bundle ${\cal V}$ of the fiber bundle ${\cal Z}\to {\cal M}_g$.
Let as before $\psi_j:{\cal Z}\to {\cal C}$ be the projection onto the $j$-th factor. Then 
we have ${\cal V}=\oplus \psi_i^*\nu$ and therefore ${\cal V}=\oplus L_i$ where
$L_i=\psi_i^*\nu$ is a complex line bundle with flat circle subbundle by naturality under
pull-back. As a consequence, there is a homomorphism 
from the orbifold fundamental group $\Gamma$ of ${\cal Z}$ into
$({\rm Top}^+(S^1))^{2g-2}$ which determines a torus subbundle of the bundle ${\cal V}$. 
By equivariance of the projections $\psi_j$ under the action of the group 
$\mathfrak{S}_{2g-2}$ by fiberwise permutations, this homomorphism can be chosen
to be equivariant under the action of $\mathfrak{S}_{2g-2}$ and hence it induces a homomorphism
\[\Theta:\mathfrak{S}_{2g-2}\ltimes \Gamma\to \mathfrak{S}_{2g-2}\ltimes
({\rm Top}^+(S^1))^{2g-2}.\]

Viewing the group $\mathfrak{S}_{2g-2}\ltimes \Gamma$ as the orbifold fundamental group of 
${\cal X}={\cal Z}\times E\mathfrak{S}_{2g-2}/\mathfrak{S}_{2g-2}$ whose rational cohomology 
coincides with the rational cohomology of ${\cal D}$, 
it now suffices to show that the class 
$e\in H_b^2(\mathfrak{S}_{2g-2}\ltimes ({\rm Top}^+(S^1))^{2g-2},\mathbb{Q})$
pulls back by $\Theta$ to the 
class $c_1({\cal W})\in H^2({\cal D},\mathbb{Q})$. This in turn holds true if
the pull-back of $e$ to 
${\cal Z}$ defines the first Chern class of the fiberwise tangent bundle ${\cal V}$ of ${\cal Z}$.
But as ${\cal V}=\oplus \psi_i^*\nu$, we have 
\[c_1({\cal V})=\sum_i c_1(\psi_i^*\nu).\]
Furthermore, as the circle subbundle of $\psi_i^*\nu$ is flat, Theorem 12.6 of \cite{Fr17} 
and naturality under pull-back shows that
$c_1(\psi_i^*\nu)=\psi_i^*\chi^*{\rm eu}({\rm Top}(S^1))$ where ${\rm eu}$ denotes the Euler class. 

The pull-back of 
the class 
$e\in H^2(\mathfrak{S}_{2g-2}\ltimes  {\rm Top}^+(S^1)^{2g-2},\mathbb{Q})$ 
to the finite index normal subgroup
$({\rm Top}^+(S^1))^{2g-2}$
is just the sum of 
the pull-backs of the Euler classes of the $2g-2$ factors. But this shows that
the class $c_1({\cal V})\in H^2({\cal Z},\mathbb{Q})$ equals the pull-back of the class $e$ under
the homomorphism $\Theta\circ \iota$ where $\iota:\Gamma\to \mathfrak{S}_{2g-2}\ltimes \Gamma$
denotes the inclusion homomorphism.
 The proposition now follows from Proposition \ref{controlextension}. 
\end{proof}

Since the norm of the class 
$e\in H^2(\mathfrak{S}_{2g-2}\ltimes ({\rm Top}(S^1))^{2g-2},\mathbb{Q})$ is 
bounded from above by $g-1$, we obtain as an immediate corollary

\begin{corollary}\label{milnorwood}
Let $B$ be a closed oriented 
surface of genus $h\geq 1$ and let $f:B\to {\cal D}$ be a continuous map; then
$\vert c_1({\cal W})(B)\vert \leq (2g-2)(2h-2)$.
\end{corollary}
\begin{proof}
Since $\Vert c_1({\cal W})\Vert_\infty \leq g-1$ and since the simplicial
volume $\Vert [B]\Vert_{\ell^1}$ of the surface $B$ equals $\vert 2\chi(B)\vert$ \cite{Fr17}, 
we obtain
as in Theorem 12.13 of \cite{Fr17} that
\[\vert c_1({\cal W})[B] \vert =\vert \langle c_1({\cal W}),[B]\rangle \vert
\leq \Vert c_1({\cal W})\Vert_\infty \Vert  [B] \Vert_{\ell^1} \leq (g-1)\vert 2 \chi(B)\vert.\]
This concludes the proof. 
\end{proof}

As a corollary, we obtain Theorem \ref{main} from the introduction.

\begin{corollary}\label{main2}
Let $E\to B$ be a surface bundle over a surface with Euler characteristic 
$\chi(E)$ and signature $\sigma(E)$; then
\[ 3\vert \sigma(E) \vert \leq \vert \chi(E)\vert.\]
\end{corollary}

\bigskip
\noindent
MATHEMATISCHES INSTITUT DER UNIVERSIT\"AT BONN\\
ENDENICHER ALLEE 60\\ 
53115 BONN, GERMANY

\bigskip
\noindent 
e-mail: ursula@math.uni-bonn.de
\end{document}